\documentclass{siamart1116}
\usepackage{amsmath}
\usepackage{float}
\usepackage{diagbox}
\usepackage{bbold}
\usepackage{hyperref}
\usepackage{enumitem}

\newtheorem{thm}{Theorem}
\newtheorem{rmk}{Remark}
\newtheorem{exm}{Example}

\newcommand{\rd}{\,\mathrm{d}}
\newcommand{\dx}{\Delta x}
\newcommand{\dt}{\Delta t}
\newcommand{\Dt}{\Delta t}
\newcommand{\aij}{\alpha_{i,j}}
\newcommand{\bij}{\beta_{i,j}}
\newcommand{\sspcoef}{\mathcal{C}}
\newcommand{\ceff}{\sspcoef_{eff}}
\newcommand{\F}{F_{ex}}
\newcommand{\G}{F_{im}}
\newcommand{\Gdot}{\dot{F}_{im}}
\newcommand{\DtFE}{\dt_{{FE}}}
\newcommand{\DtDW}{\dt_{{DW}}}
\newcommand{\DtSD}{\dt_{{SD}}}

\newcommand{\DtND}{\dt_{{ND}}}

\newcommand{\m}[1]{\mathbf{#1}}
\newcommand{\zero}{\m{0}}
\newcommand{\vy}{\m{y}}
\newcommand{\vb}{\m{b}}
\newcommand{\bdot}{\dot{\vb}}
\newcommand{\ve}{\m{e}}
\newcommand{\vc}{\m{c}}
\newcommand{\vcdot}{\dot{\vc}}
\newcommand{\mS}{\m{S}}
\newcommand{\mR}{\m{R}}
\newcommand{\mP}{\m{P}}
\newcommand{\mW}{\m{W}}
\newcommand{\mQ}{\m{Q}}
\newcommand{\mA}{\m{A}}
\newcommand{\mT}{\m{T}}
\newcommand{\mV}{\m{V}}
\newcommand{\mC}{\m{c}}
\newcommand{\mAh}{\m{\hat{A}}}
\newcommand{\mD}{\m{D}}
\newcommand{\mDdot}{\m{\dot{\mD}}}
\newcommand{\mAdot}{\m{\dot{\mA}}}
\newcommand{\Cdot}{\m{\dot{\mC}}}
\newcommand{\Chat}{\m{\hat{\mC}}}
\newcommand{\bhat}{\m{\hat{\vb}}}

\title{A review of high order  strong stability preserving two-derivative explicit,  implicit, and IMEX methods}

\author{%
Sigal Gottlieb\thanks{Mathematics Department, University of Massachusetts Dartmouth, North Dartmouth, MA 02747. Email: sgottlieb@umassd.edu.
SG's research was supported in part by AFOSR Grant No. FA9550-23-1-0037, DOE  Grant No. DE-SC0023164 Subaward RC114586,
and Mass Dartmouth’s Marine and Undersea Technology (MUST) Research Program funded by the ONR Grant No. N00014-20-1-2849.} \and
Zachary J. Grant\thanks{Mathematics Department, University of Massachusetts Dartmouth, North Dartmouth, MA 02747. Email: zgrant@umassd.edu.
ZJG's research was supported in part by  DOE grant  No. DE-SC0023164 Subaward RC114586.}
}

\begin{document}
\maketitle


\bibliographystyle{siam}

\begin{abstract}  
High order strong stability preserving (SSP)  time discretizations  ensure the nonlinear non-inner-product 
strong stability properties of  spatial discretizations suited for the stable simulation of hyperbolic PDEs. 
Over the past decade multiderivative time-stepping  have been used for the time-evolution hyperbolic PDEs, 
so that the  strong stability properties of these methods have become increasingly relevant.
In this work we review sufficient  conditions for a two-derivative multistage method to 
preserve the strong stability properties of spatial discretizations in a forward Euler
and different conditions on the second derivative.
In particular we present the  SSP  theory for explicit and  implicit two-derivative 
Runge--Kutta schemes, and discuss a special condition on the second derivative  under which these
implicit methods may be unconditionally SSP. This condition is then used in the context of 
 implicit-explicit (IMEX) multi-derivative  Runge--Kutta schemes, where the time-step restriction is  independent of the stiff term.
 Finally, we present  the SSP theory for implicit-explicit (IMEX) multi-derivative general linear methods,
 and some novel second and third order methods  where the time-step restriction is  independent of the stiff term. 
  \end{abstract}
  
\section{Overview} 
Strong stability preserving Runge--Kutta  methods were developed by Shu in \cite{shu1988,shu1988b}
to  preserve the nonlinear stability properties of forward Euler in any norm, semi-norm, or convex functional. 
This approach was  further studied for Runge--Kutta, multistep, and general linear methods in
\cite{gottlieb1998, gottlieb2001, ruuth2001, shu2002, spiteri2002, spiteri2003, gottlieb2003, hundsdorfer2003, 
ruuth2004, hundsdorfer2005, ruuth2006, spijker2007, ketcheson2008, ketcheson2007, liu2008, SSPbook2011,
nguyen2014strong,Izzo2015,IZZO2020206,BRAS2021113612}.
The study of the SSP properties of different time-stepping methods has been aided by its connections to 
contractivity theory \cite{ferracina2005a,  ferracina2005,higueras2004a, higueras2005a,ferracina2008,
ketcheson2007, ketcheson2008}.  
SSP methods have proven useful in the solution of hyperbolic PDEs
using many different spatial approaches   \cite{cockburn2004,
peng1999,caiden2001,enright2002,cheng2003,cockburn2005,jin2005,
caiden2001,delzanna2002,baiotti2005, bassano2003,carrillo2003,
tanguay2003,feng2004,labrunie2004,balbas2005,zhang2006,pantano2007,
sun2006,cheruvu2007,wang2005,wang2007a}.  They have been 
widely used for many application areas \cite{wang2005, patel2005,sun2006,
caiden2001,bassano2003,delzanna2002,baiotti2005,zhang2006,feng2004,
balbas2005,mignone2005,labrunie2004,cheruvu2007,pantano2007,cockburn2004,
carrillo2003,tanguay2003,cockburn2005,cheng2003,jin2005}.
 
More recently, two-derivative SSP methods have become a subject of increasing interest
\cite{Nguyen-Ba2010,SDpaper,TSpaper,MORADI2019,EIS2dSSP2020,Jingwei2022,QIN2024106089,QIN2023}.
In this work we review two-step Runge--Kutta methods which preserve the properties of forward Euler and a selection of 
conditions on the second derivative. 
In Section \ref{sec:SSP} we review SSP theory for Runge--Kutta methods. 
In Section \ref{sec:RK-2D}  we discuss the strong stability preservation theory for two derivative methods
and propose three different conditions on the second derivative that each lead to SSP two-derivative
Runge--Kutta methods, and present some optimal methods in each class. Finally, in Section \ref{sec:2D-RK-IMEX}
we discuss implicit-explicit (IMEX) Runge--Kutta methods with two derivatives treated implicitly, and  in 
Section \ref{sec:2D-GLM-IMEX} we extend this approach to IMEX general linear methods.
We note that while this paper is a review of the topic of SSP two derivative methods,
 the material in Section \ref{sec:2D-GLM-IMEX}  and the related order conditions in Appendix 
\ref{sec:OC-IMEXGLM} is new.

\section{SSP methods}   \label{sec:SSP}
In numerically solving the  hyperbolic conservation law
\begin{eqnarray}
\label{pde}
	U_t + f(U)_x = 0,
\end{eqnarray}
oscillations leading to instability may occur when  the exact solution develops sharp gradients or discontinuities. 
High order spatial discretizations  that can handle discontinuities while preserving 
nonlinear non-inner-product stability properties, such as total variation  stability or positivity, are required for the
stable simulation of such problems.  
After discretizing  in space using such a specially designed scheme (e.g. DG, TVD, WENO),
 we obtain the  semi-discretized equation
\begin{eqnarray}
\label{ode}
u_t = F(u),
\end{eqnarray}
(where $u$ is a vector of approximations to $U$)
that has the property that the numerical solution is strongly stable when coupled with forward Euler time stepping
\begin{eqnarray}  \label{FEcond}
\|u^{n+1} \| &=& 	\| u^n + \dt F(u^{n})  \| \leq \| u^n \|, \quad  0  \leq \dt \leq \DtFE,
\end{eqnarray}
where $\| \cdot \|$ is some norm, semi-norm, or convex functional, depending on the design of the spatial discretization.

In practice, Euler's method is not a preferred method, as it is low order and has a linear stability region
that excludes the imaginary axis. Instead, we desire a  higher order method that preserves the  strong stability  property
$	\| u^{n+1} \| \le \|u^n\| $
under a (possibly modified) time-step restriction $\dt \leq \sspcoef \DtFE$.
If such a method exists for $ \sspcoef >0$, we call it a {\em strong stability preserving} (SSP) method, and we 
say that  $ \sspcoef$ is the {\em SSP coefficient} of the method.
The research in the field of SSP methods focuses on finding SSP methods of high order with largest possible $\sspcoef$.

\subsection{Explicit SSP Runge--Kutta methods}
In this section we will show that if a higher order Runge--Kutta method 
can be written as convex combinations of forward Euler steps, 
then any  convex functional  property \eqref{FEcond} satisfied by the forward Euler scheme 
\begin{equation}  \label{FE}
u^{n+1}  =  u^n + \dt F(u^{n}) 
\end{equation}
will be {\em preserved} under a modified  time-step restriction $\dt \leq \sspcoef \DtFE$  \cite{shu1988,shu1988b}. 
In this sense, the forward Euler method is the building block to constructing SSP Runge--Kutta methods.
In fact, the forward Euler condition is a natural and important property of an operator $F$;
it has been noted \cite{ferracina2004}  that it is equivalent to the circle condition which is central to the analysis 
of contractive functions $F$.

The $s$-stage explicit Runge--Kutta method  
\begin{eqnarray}
\label{rkSO}
y^{(0)} & =  & u^n, \nonumber \\
y^{(i)} & = & \sum_{j=0}^{i-1} \left( \aij y^{(j)} +
\dt \bij F(y^{(j)}) \right), \; \; \; \; i=1, . . ., s\\
 u^{n+1} & = & y^{(s)}  \nonumber
\end{eqnarray} 
can be rewritten as  convex combination of forward Euler steps of the form  \eqref{FE} by factoring each stage
\begin{eqnarray*}
y^{(i)} & = & \sum_{j=0}^{i-1} \aij \left(  y^{(j)} +
\dt \frac{\bij}{\aij} F(y^{(j)}) \right).
\end{eqnarray*} 
If all the coefficients $\aij$ and $\bij$ are non-negative, and  $\aij$ is zero only if 
its corresponding $\bij$ is zero, then the  consistency condition  $\sum_{j=0}^{i-1} \aij =1$
and the forward Euler condition \eqref{FEcond} imply that
each stage is bounded by
\begin{eqnarray*}
\| y^{(i)}\| & \leq  &  
   \sum_{j=0}^{i-1} \aij  \, \left\| y^{(j)} + \dt \frac{\bij}{\aij} F(y^{(j}) \right\|  \leq \| y^{(j)} \|
\end{eqnarray*}
for $ \frac{\bij}{\aij}  \dt \leq \DtFE$.
Putting this together for the entire Runge--Kutta method  \eqref{rkSO},
we see that 
\begin{eqnarray}
	\label{rkSSP}
	 \| u^{n+1}\| \leq \| u^{n} \|  \; \; \; \; \; \; \; \mbox{for} \; \; \; \; 
	\dt \leq \sspcoef  \DtFE \; \; \; \;  \mbox{where} \; \; \; \;  \sspcoef = \min_{i,j}  \frac{\aij}{\bij}.
\end{eqnarray}
(If any $\beta$ is equal to zero, we consider that ratio to be infinite.)  

The resulting time-step restriction is a combination of two
distinct factors: (1) the term $\DtFE$ that depends on the spatial discretization, and 
(2) the SSP coefficient $\sspcoef$ that depends only on the time-discretization. As stated above, any method
that admits such a decomposition with $\sspcoef>0$ is called a {\em strong stability preserving (SSP)}
method.

This convex combination  decomposition was used in the development of 
 second and third order explicit SSP Runge--Kutta methods \cite{shu1988} 
 and later of fourth order SSP Runge--Kutta methods methods \cite{spiteri2002, ketcheson2008}.
These methods not only guarantee the  strong stability properties of any  spatial discretization, 
given only the forward Euler condition, but also ensure that 
the intermediate stages in a Runge--Kutta method  satisfy the strong stability property as well.
The convex combination decomposition is not only a  sufficient condition,
it has been shown to be necessary as well
 \cite{ferracina2004, ferracina2005,higueras2004a, higueras2005a}.

Much research on SSP methods focuses on finding high-order time discretizations
with the largest allowable time-step  $\Dt \le \sspcoef \DtFE$ by maximizing
the   {\em SSP coefficient} $\sspcoef$ of the method. 
In fact, a more relevant measure is the
 {\em effective SSP coefficient} $\ceff = \frac{\sspcoef}{s}$ where the cost is 
 relative to the number of function evaluations
at each time-step -- typically the number of stages  $s$ of a method.
All explicit  Runge--Kutta methods with positive SSP coefficient
have a tight bound on the effective SSP coefficient:  $\ceff  \leq 1$ \cite{SSPbook2011}.

Furthermore, it has been shown that explicit  Runge--Kutta methods with positive SSP coefficients
suffer from an order barrier: they cannot be more than
fourth-order accurate \cite{kraaijevanger1991,ruuth2001}.
These bounds and barriers  on explicit  SSP Runge--Kutta methods 
drive the  study of other classes of SSP methods, 
such as explicit or implicit methods with multiple stages, steps, and/or derivatives.
Explicit multistep SSP methods of order  $p>4$ do exist, but have 
severely restricted time-step requirements \cite{SSPbook2011}. Explicit multistep
multistage methods that are SSP and have order $p>4$ have been developed as well \cite{tsrk,msrk}.
This review paper focuses on methods that include a second derivative terms.
The analysis of explicit two-derivative Runge--Kutta methods will be discussed in 
Sections \ref{sec:SD} and \ref{sec:TS}.

 \subsection{Implicit SSP Runge--Kutta methods}
One approach to alleviating  the bounds and barriers of  explicit  methods  is to turn to implicit methods. 
While implicit SSP Runge--Kutta methods exist up to order $p=6$, 
they suffer from a step-size restriction   that is quite severe.
For implicit methods the SSP coefficient is usually bounded by twice the number of stages for a Runge--Kutta 
 method \cite{KetchDW}. This is true for all implicit methods that have been tested: Runge--Kutta, multistep methods,
 and general linear methods. Although this bound on the effective SSP coefficient is twice the maximal size of 
 the bound on the explicit method, the additional computational cost for the implicit solve far outweighs the benefits.
 
Once way of overcoming the bound on the SSP coefficient is 
by using an additional operator $\tilde{F}$ that approximates the same spatial operator as $F$ but
satisfies a downwind first order condition
\begin{eqnarray} \label{DWcond}
\hspace{-1in} \mbox{\bf Downwind condition:}&  \nonumber \\
 &  \| u - \dt \tilde{F}(u)\| \leq \| u \|  \; \; \;   \mbox{ for all } \;\;  \dt \leq  \DtDW, 
 \end{eqnarray}
 instead of the usual forward Euler condition \eqref{FEcond}. Such an operator can often be defined
 when solving hyperbolic PDEs. 
By incorporating the downwind operator $ \tilde{F}$ as well as the usual operator $F$ into an implicit Runge--Kutta method,
Ketcheson and his students designed  families of  second order and third order methods that are 
unconditionally SSP  \cite{KetchDW,YiannisDW}. The second order methods \cite{KetchDW}
have coefficients that depend on $r$
\begin{eqnarray*}
y^{(1)} & = & \frac{2}{r(r-2)} u^n +  \frac{2}{r} \left(  y^{(1)}  + \frac{1}{r} \dt F(y^{(1)})  \right)  
+  \frac{r^2 - 4 r + 2}{r(r-2)}  \left(  y^{(2)}  +  \frac{1}{r}  \dt \tilde{F}(y^{(2)})  \right)  \\
y^{(2)} & = &   y^{(1)}  + \frac{1}{r} \dt F(y^{(1)})     \\
 u^{n+1} & = & y^{(2)}  + \frac{1}{r} \dt F(y^{(2)})     .
\end{eqnarray*}
For $r>2 + \sqrt{2}$, this family of methods is A-stable and SSP with  $\sspcoef = r$.
Since $r$ can be chosen to be arbitrarily large, these methods can be SSP for arbitrarily large $\sspcoef$.
The second and third order methods are {\em fully implicit} and so require the 
simultaneous solution of all the stages.
 
 Inclusion of the downwind term $\tilde{F}$ allows us to bypass the  requirement that all coefficients in the scheme are 
 non-negative. This strict requirement leads to barriers and bounds on the allowable order and SSP coefficient. 
 Allowing negative coefficients provides added flexibility which alleviates the barriers and bounds. Similarly,  
 another way of overcoming the bound on the SSP coefficient involves the inclusion of a second derivative
 and will be discussed in Sections  \ref{sec:ND}, \ref{sec:2D-RK-IMEX}, and  \ref{sec:2D-GLM-IMEX}.

\section{SSP  two-derivative Runge--Kutta methods} \label{sec:RK-2D}
Enhancing  Runge--Kutta methods with additional derivatives was proposed in 
\cite{Tu50,StSt63}, and  multistage multiderivative time integrators for ordinary differential 
equations were studied in  \cite{shintani1971,shintani1972,KaWa72,KaWa72-RK,
mitsui1982,ono2004, tsai2010}. In the last decade multistage multiderivative methods
were applied to the time evolution of  partial differential equations  \cite{sealMSMD2014,tsai2014,LiDu2016a,LiDu2016b,LiDu2018}.
In particular, for hyperbolic PDEs, adding a second derivative to  Runge--Kutta methods is efficient  because the computation of 
the Jacobian of the flux $f(u)$ in \eqref{pde} is generally needed for the stable evolution of the equation; 
the second derivative computation $\dot{F}$ relies on the computation of this 
Jacobian as well. For this reason, we limit our discussion to methods with at most two derivatives: $F$ and $\dot{F}$.
Two derivative Runge--Kutta methods take the form
\begin{eqnarray}
\label{MSMD}
y^{(i)} & = &  u^n +  \dt \sum_{j=1}^{s}  a_{ij} F(y^{(j)}) +
\dt^2   \sum_{j=1}^{s} \dot{a}_{ij} \dot{F}(y^{(j)}) , \; \; \; \; i=1, . . ., s \\
 u^{n+1} & = &  u^n +  \dt \sum_{j=1}^{s} b_{j} F(y^{(j)}) +
\dt^2  \sum_{j=1}^{s} \dot{b}_{j} \dot{F}(y^{(j)})  . \nonumber
\end{eqnarray}
The method can be written in matrix form
\begin{eqnarray}
\label{MSMDmatrix}
\vy & = &  u^n +  \dt \mA F(\vy) +
\dt^2 \mAdot \dot{F}(\vy)  \\
 u^{n+1} & = &  u^n +  \dt \vb^T F(\vy) + \dt^2 \bdot^T \dot{F}(\vy)  \nonumber
\end{eqnarray}
where
$ \mA_{ij}  =  a_{ij} $ , 
$  \mAdot = \dot{a}_{ij}$, 
 and $b$ and $\dot{b}$ are column vectors with the elements $b_{j}$ and $\dot{b}_{j}$, respectively.
We give the order conditions for this method in Appendix \ref{OC-MDRK}.

While the forward Euler condition is central to the development of SSP methods,
additional conditions on the derivative allow us to devise SSP two-derivative Runge--Kutta 
methods. The exact form of the  condition 
depends on the types of the problems we aim to solve: we considered (1) a condition that depends on an evolution
of the second derivative; (2) a condition that mimics the explicit Taylor series method; 
and (3) a condition that is inspired by the implicit Taylor series method.
These three conditions, and the SSP methods that arise from them, are the subject of the next subsections.

\begin{rmk}
We note that for hyperbolic problems, we use a  typical method-of-lines approach where
 the operator  $F$  is a spatial discretization of the term $U_t= -f(U)_x$ 
that leads to the system $u_t = F(u)$. 
In principal,  the computation of the second derivative term $\dot{F}$ should follow directly from the definition of $F$,
so that  $\dot{F} = F(u)_t = F_u u_t = F_u F$. However, in practice, such a computation may be expensive.
Instead, we use a Lax-Wendroff type approach to compute $\dot{F}$. We  go back to the original  PDE
\eqref{pde},  replace the time derivatives by the spatial derivatives, and discretize these in space. 
\end{rmk}

\subsection{Second derivative condition} \label{sec:SD}
In any method-of-lines formulation \eqref{ode}, the  spatial discretization $F$ is designed to satisfy 
the {\em forward Euler} condition \eqref{FEcond}
\begin{eqnarray*} 
	\mbox{\bf \; \; \; Forward Euler condition} \; \; \; \; \; \;  \|u^n +\Delta t F(u^n) \| \leq \| u^n  \| 
	\; \; \;  \mbox{for}  \; \; \;   \Delta t \leq \Delta t_{FE},
\end{eqnarray*}
where $\| \cdot \|$ denotes the desired norm, semi-norm, or convex functional.
To account for the second derivative in \eqref{MSMD}, we  impose a similar strong stability condition on the second derivative
\begin{eqnarray} \label{SDcond}
& & \mbox{\bf \; \; Second derivative condition} \nonumber  \\
& & \hspace{0.5in} \|u^n +\Delta t^2 \dot{F}(u^n) \| \leq \| u^n  \| \; \; \; \mbox{for}  \; \;   \Delta t \leq  \DtSD = K \DtFE.
\end{eqnarray}
We find it useful to define the constant $K =  \DtSD/\DtFE$ 
 that compares the stability condition of the second derivative term
to that of the forward Euler term, so that condition \eqref{SDcond}
holds for $ \Delta t \leq  K  \Delta t_{FE}.$


The choice of second derivative condition \eqref{SDcond} over the more natural Taylor series condition
\eqref{TScond} was motivated by  the unique  
two-stage two-derivative fourth order method  \cite{LiDu2016a, LiDu2016b, LiDu2018}
\begin{align}  \label{2s4p}
    y^{(1)}    &= u^n + \frac{\Delta t}{2} F(u^n) + \frac{\Delta t^2}{8} \dot{F}(u^n) \nonumber \\
    u^{n+1} &= u^n + \Delta t           F(u^n) + \frac{\Delta t^2}{6}( \dot{F}(u^n)+2\dot{F}( y^{(1)} )).
\end{align}
This method  is SSP if $F$ satisfies the forward Euler 
condition \eqref{FE} and  $\dot{F}$ satisfies the second derivative condition
\eqref{SDcond}. This is because we write  each stage of \eqref{2s4p} as a convex combination of the building blocks  in
\eqref{FE} and \eqref{SDcond}. The first stage can be written  for $0 \leq \alpha \leq 1$
\begin{eqnarray*}
  y^{(1)}   & = & \alpha \left( u^n + \frac{\Delta t}{2\alpha} F(u^n)\right)
   + (1-\alpha) \left( u^n +  \frac{\Delta t^2}{8(1-\alpha)} \dot{F}(u^n) \right),
\end{eqnarray*}
where $\| y^{(1)}  \| \leq \| u^n  \|$ for an appropriate time-step 
\[ \dt \leq \min \left\{ 2\alpha \DtFE, \sqrt{8(1-\alpha)} K  \DtFE \right\}\]
dictated by \eqref{FEcond} and \eqref{SDcond}. 
Similarly, for  $0 \leq \alpha \leq 1$, $0 \leq \beta \leq 1$, $0 \leq 1 - \alpha -  \beta \leq 1$, the second stage 
\begin{eqnarray*}
 u^{n+1}     &= & (1 - \alpha - \beta)  \left( u^n   + \dt  \frac{2-\alpha}{2(1-\alpha - \beta)}  F(u^n)  \right) 
        + \beta \left( u^n +  \frac{4 - 3\alpha}{24 \beta} \Delta t^2 \dot{F}(u^n)\right)  \\
     && + \alpha \left( y^{(1)} + \frac{\Delta t^2}{3\alpha}  \dot{F}( y^{(1)} )\right) 
\end{eqnarray*}
 is SSP for an appropriate time-step.

While not every multistage multiderivative method is SSP in the sense that it preserves the 
properties of \eqref{FE} and \eqref{SDcond}, we can use a convex decomposition approach 
to determine which ones are, and to find the value of $\dt$ for which we can ensure the method satisfies the
desired strong stability properties.  This approach is generalized in the following theorem, which also suggests
the optimal decomposition of the method.

\begin{thm} [\cite{SDpaper}]
\label{thmSD}
Given spatial discretizations $F$ and $\dot{F}$ that satisfy \eqref{FE} and \eqref{SDcond},
a two-derivative multistage method  of the form \eqref{MSMDmatrix} preserves  the strong stability property 
$ \| u^{n+1} \| \leq \|u^n \|$  under the time-step restriction $\Delta t \leq r \DtFE$ for some $r>0$,
if  satisfies the  component-wise inequalities
\begin{subequations} 
\begin{align} 
\left( I +r \mS  + \frac{r^2}{K^2} \dot{\mS} \right)^{-1}  \ve  \geq  0   \label{SSPcondition1} \\
r  \left( I +r \mS  + \frac{r^2}{K^2} \dot{\mS} \right)^{-1} \mS \geq  0   \label{SSPcondition2} \\
 \frac{r^2}{K^2}   \left( I +r \mS  + \frac{r^2}{K^2}  \dot{\mS} \right)^{-1} \dot{\mS} \geq 0   \label{SSPcondition3}
 \end{align} 
\end{subequations}
where 
\[ \mS= \left[ \begin{array}{ll} \mA & \zero_{s \times 1} \\ \vb^T & 0 \end{array} \right] \; \; \; \; \;
	\mbox{and} \; \; \; \; \; \dot{\mS}= \left[ \begin{array}{ll} \dot{\mA} & \zero_{s \times 1}  \\ \dot{\vb}^T & 0  \end{array}  \right]
\] and $\ve$ is a vector of ones.
\end{thm}

\begin{exm} {\bf Motivating Example:}
An example in which conditions \eqref{FE} and \eqref{SDcond} are satisfied and the SSP property is desired 
was considered in \cite{SDpaper}.
Consider  the simple linear one-way wave equation 
\[ U_t = U_x.\]
We can semi-discretize in space to obtain $u_t = F(u)$ where $F$ is defined by a first-order upwind method
\[ F(u^n)_j := \frac{1}{\dx} \left(u^n_{j+1} - u^n_j \right) \approx U_x( x_j ).\] 

If $u$ is sufficiently smooth, the second derivative in time is also the second derivative in space:
\[ U_{tt} = (U_x)_t = ( U_{t} )_x = U_{xx}.\]
This convenient fact allows us to use a  Lax-Wendroff type approach to define $\dot{F}(u^n)$ by a centered spatial discretization of $U_{xx}$,
e.g.
\[ \dot{F}(u^n)_j := \frac{1}{\dx^2} \left( u^n_{j+1}- 2 u^n_j + u^n_{j-1} \right) \approx  U_{xx}( x_j ).\]

Both $F$ and $\dot{F}$ satisfy the  total variation diminishing (TVD) property:
\[ \left\| u^n + \Delta t F(u^n) \right\|_{TV} \leq   \left\| u^n\right\|_{TV}  \; \;  \; \;  \mbox{for} \; \; \; \; \Delta t \leq \Delta x,  \]
and
\[ \left\| u^n +  \Delta t^2 \dot{F}(u^n) \right\|_{TV} \leq   \left\| u^n\right\|_{TV}  \; \;  \; \;  \mbox{for} \; \; \; \; \Delta t \leq \frac{\sqrt{2}}{2} \Delta x.  \]
In Section \ref{sec:numerical} we show how the methods that satisfy the conditions of Theorem \eqref{thmSD} perform for these
$F$ and $\dot{F}$.
\end{exm}

\subsubsection{Optimal methods based on the forward Euler and second derivative conditions}
\label{sec:FE-SDmethods}

In this section we present some optimal methods (in the sense of the largest $\sspcoef$) that satisfy 
Theorem \eqref{thmSD} for  
 spatial discretizations $F$ and $\dot{F}$ that satisfy \eqref{FEcond} and \eqref{SDcond}.

\noindent{\bf Optimal one stage second order methods:}
There is a unique explicit one-stage  two-derivative second order method: the Taylor series method
\begin{eqnarray} \label{TSmethod}
 u^{n+1} &=& u^n + \dt F(u^n) + \frac{1}{2} \dt^2 \dot{F}(u^n) .
 \end{eqnarray}
The optimal decomposition of this method, and the corresponding  SSP coefficient, depend on the  value of $K$ in \eqref{SDcond}.
This method can be written as a convex combination of two terms that satisfy the conditions \eqref{FE} and \eqref{SDcond}:
\[ u^{n+1} = (1-\alpha) \left( u^n + \frac{1}{1-\alpha} \dt F(u^n)  \right) + 
\alpha \left( u^n + \frac{1}{2 \alpha} \dt^2 \dot{F}(u^n) \right) .\]
This is SSP for $\dt \leq \max \{ (1-\alpha) \DtFE, \sqrt{2 \alpha} K  \DtFE \}$, so we set
\[ (1-\alpha)^2 =  2 \alpha K^2 \; \; \; \implies \; \; \; \;  \alpha = 1 + K^2 \pm K \sqrt{ 2 + K^2} \]
to obtain
\[ \sspcoef =  K \sqrt{2 + K^2} - K^2.\] 

%

\noindent{\bf Two-stage  third order methods.} 
Optimal SSP  explicit two-stage two-derivative third order methods take the form
\begin{eqnarray} \label{2s3p}
u^*     &=& u^n+a \Delta t F(u^n)+ \hat{a} \Delta t^2 \dot{F}(u^n), \nonumber \\
u^{n+1} &=& u^n +b_1  \Delta t  F(u^n) +  b_2 \Delta t  F(u^*) +\hat{b_1}
    \Delta t^2  \dot{F}(u^n) + \hat{b_2} \Delta t^2  \dot{F}(u^*).
\end{eqnarray}
The optimal method has  SSP coefficient $\sspcoef = r$ given by  
the smallest positive root of the polynomial
\[  2 K (a_0 - 2K) + 4 K^3 a_0  - a_0 r 
+ \frac{1-a_0}{2 K^2} r^2 
-\frac{\frac{a_0}{2K} + K}{6 K^3} r^3 , \]
for $a_0 =\sqrt{K^2+2} - K$.
The coefficients of the optimal method for each $K$ depend on $K$ and $r$ and are given by

\begin{eqnarray*}
&a =  \frac{1}{r} \left(K a_0 \right), \; \; \; \; \; \; \; \; \; \;  b_1  =  1- b_2,  \; \; \; 
&b_2  =  \frac{2 K^2 (1 - \frac{1}{r} ) + r } {  K a_0 + 2 K^2} - \frac{r^2}{3 K^2} ,\\
&\hat{a}  =  \frac{1}{2} a^2, \; \; \; \; 
\hat{b}_1  =  \frac{1}{2} - \frac{1}{2} a b_2 - \frac{1}{6 a} ,\; \; \; \; \;
&\hat{b}_2  =  \frac{1}{6 a} - \frac{1}{2} a b_2 . 
\end{eqnarray*}
%

\noindent{\bf Fourth order methods.} 
The  two-stage two-derivative fourth order method is given in \eqref{2s4p}.
Although the method is unique, the optimal decomposition, and therefore the  SSP coefficient, 
depends on $K$. The SSP coefficient $\sspcoef=r$ is given by the smallest positive  root of the polynomial:
\[  r^4 + 4 K^2 r^3 -12 K^2 r^2 - 24 K^4 r + 24 K^4.\] 
Although the method can be implemented in its usual form, for analysis purposes the 
optimal Shu-Osher decomposition may be helpful:
\begin{eqnarray*}
y^{(1)}  & = & \left( 1- \frac{4 r K^2 + r^2}{8 K^2} \right) u^n + \frac{r}{2} \left( u^n + \frac{\dt}{r} F(u^n) \right)
+ \frac{ r^2}{8 K^2}  \left( u^n + \frac{K^2}{r^2} \dt^2 \dot{F}(u^n) \right) \\
u^{n+1} &=& r \left( 1 - \frac{r^2}{6K^2} \right) \left( u^n + \frac{\dt}{r} F(u^n) \right) + 
	\frac{r^2 ( 4 K^2-r^2)}{24 K^4}  \left( u^n + \frac{K^2}{r^2} \dt^2 \dot{F}(u^n) \right) \\
	&& \;  + \; \; \; 
	\frac{r^2}{3 K^2}  \left( y^{(1)} + \frac{K^2}{r^2} \dt^2 \dot{F}(y^{(1)}) \right).  \nonumber
\end{eqnarray*}
Increasing the number of stages to three allows for  fourth-order SSP methods with 
a larger SSP coefficient, as  described in \cite{SDpaper}.  Increasing the number
of stages also allows for fifth order.

\noindent{\bf Three-stage  fifth order methods.} 
While explicit Runge--Kutta methods have an  order barrier of $p=4$  \cite{kraaijevanger1991,ruuth2001}, 
 including a second derivative term allows us to achieve fifth order SSP methods.
The optimal SSP three stage two-derivative fifth order methods were given in \cite{SDpaper}:
\begin{eqnarray} \label{3s5p}
    y^{(1)}     &= & u^n +  a_{21} \Delta t F(u^n) +  \dot{a}_{21}\Delta t^2  \dot{F}(u^n) \nonumber \\
     y^{(2)} &=& u^n +  a_{31} \Delta t F(u^n)  
    +  \dot{a}_{31}\Delta t^2  \dot{F}(u^n)  +  \dot{a}_{32}\Delta t^2  \dot{F}(y^{(1)} )  \\
    u^{n+1} &=& u^n +  \Delta t  F(u^n)  
    + \Delta t^2 \left( \dot{b}_1 \dot{F}(u^n)+ \dot{b}_2 \dot{F}(y^{(1)} ) + \dot{b}_3 \dot{F}(y^{(2)} ) \right). \nonumber
\end{eqnarray}
The SSP coefficient $\sspcoef = r$ is the  largest positive root of 
\[  10 r^2 a_{21}^4  - (100 K^2 + 10 r^2) a_{21}^3   +( 130 K^2  + 3 r^2 ) a_{21}^2   
- 50 K^2 a_{21}+ 6 K^2, \]
where 
\[  a_{21} =\frac{K^6}{r^6} \left( - \frac{2}{K^4} r^5 + \frac{10}{K^4} r^4  +  \frac{40}{K^2} r^3
   - \frac{120}{K^2} r^2  -240 r + 240 \right).
\]
 Given $K$, we can find the corresponding $r$, and the coefficients are then given as a one-parameter system 
\begin{eqnarray*}
\dot{a}_{21} &=& \frac{1}{2}  a_{21}^2 , \; \; \; \; \;   a_{31}    =   \frac{3/5 -a_{21}}{1-2 a_{21}}, \\ 
\dot{a}_{32}   & = & \frac{1}{10} \left(  \frac{(\frac{3}{5} -a_{21})^2}{a_{21} (1-2 a_{21})^3} - 
\frac{\frac{3}{5} -a_{21}}{(1-2 a_{21})^2}  \right), \; \; \; \; \; 
\dot{a}_{31}   =   \frac{1}{2} \frac{(\frac{3}{5} -a_{21})^2}{(1-2 a_{21})^2}  -\dot{a}_{32}, \\
\dot{b}_2  & = & \frac{2 a_{31}-1}{12 a_{21}(a_{31}-a_{21})} , \; \; \; \; \; 
\dot{b}_3   =  \frac{1-2 a_{21}}{12 a_{31} (a_{31}-a_{21})} , \; \; \; \; \; 
\dot{b}_1   =  \frac{1}{2} - \dot{b}_2 - \dot{b}_3. \\
\end{eqnarray*}
This method shows that  explicit multiderivative Runge--Kutta methods  break the well-known fourth order barrier for 
explicit SSP Runge--Kutta methods.

\subsubsection{Numerical example}  \label{sec:numerical}
To demonstrate how these numerical methods perform on our motivating example, 
we repeat the example from \cite{SDpaper}.
We simulate  the linear advection equation
\[ U_t = U_x, \; \; \;  \mbox{on the domain} \; \; \; x \in [0,1]  \] 
using a   first order finite difference for the first derivative and 
a second order centered difference for the second derivative 
\[
	F(u^n)_j := \frac{u^n_{j+1}-u^n_j}{\Delta x} \approx U_x( x_j ), \; \; \; \; \; \;
		\mbox{and} \; \; \; \; \; \;
	\dot{F}(u^n)_j := \frac{u^n_{j+1}- 2 u^n_j + u^n_{j-1}}{\Delta x^2} \approx  U_{xx}( x_j ).
\]
These spatial discretizations satisfy \eqref{FEcond} with $\DtFE = \Delta x$, and
\eqref{SDcond} with $K  = \frac{1}{\sqrt{2}} $.
We use a step function  initial condition 
\begin{eqnarray*}
	u_0(x) = \left\{ \begin{array}{ll}
		1 & \text{if}\ \frac{1}{4} \leq x \leq \frac{1}{2}, \\
		0 & \text{otherwise},
\end{array} \right. 
\end{eqnarray*}
and periodic boundary conditions $u(0,t) = u(1,t)$. 

\begin{figure}[t!] \hspace{0.25in} 
\includegraphics[width=0.45\textwidth]{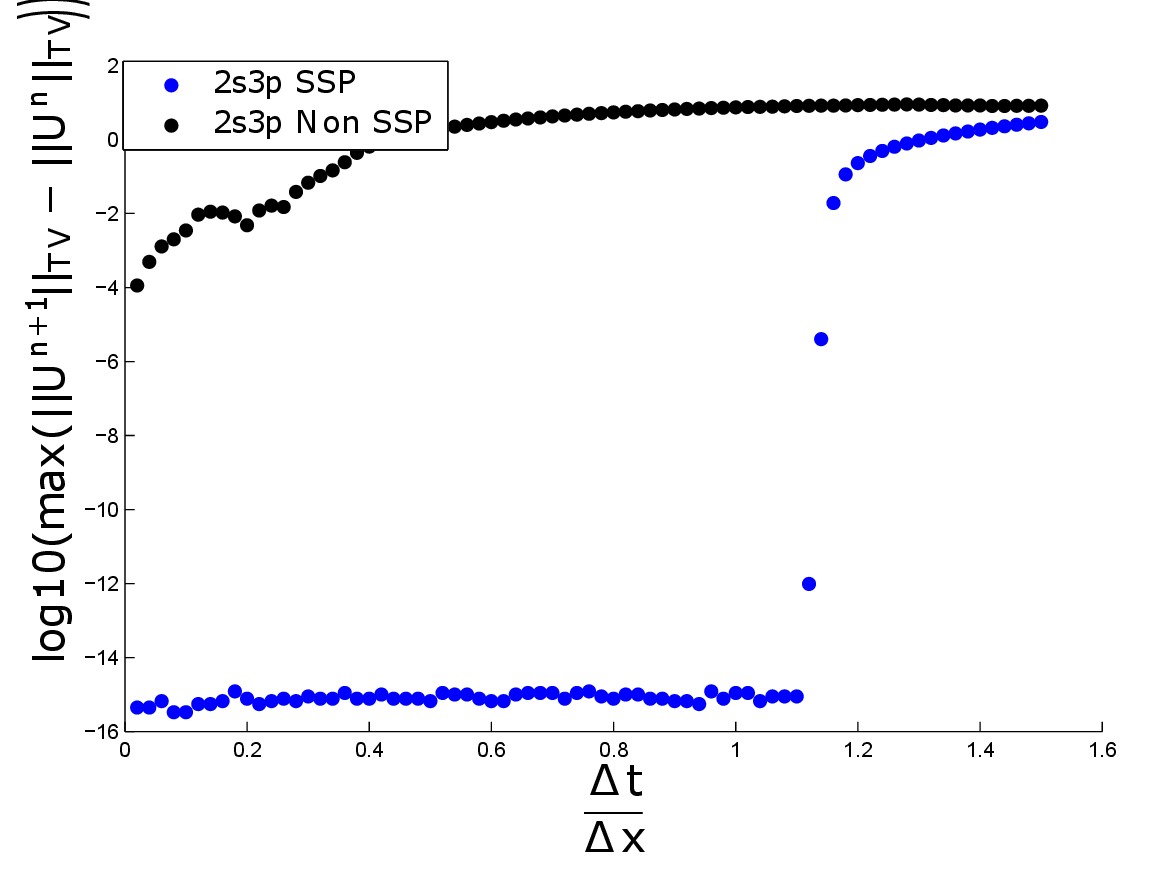}
\includegraphics[width=0.45\textwidth]{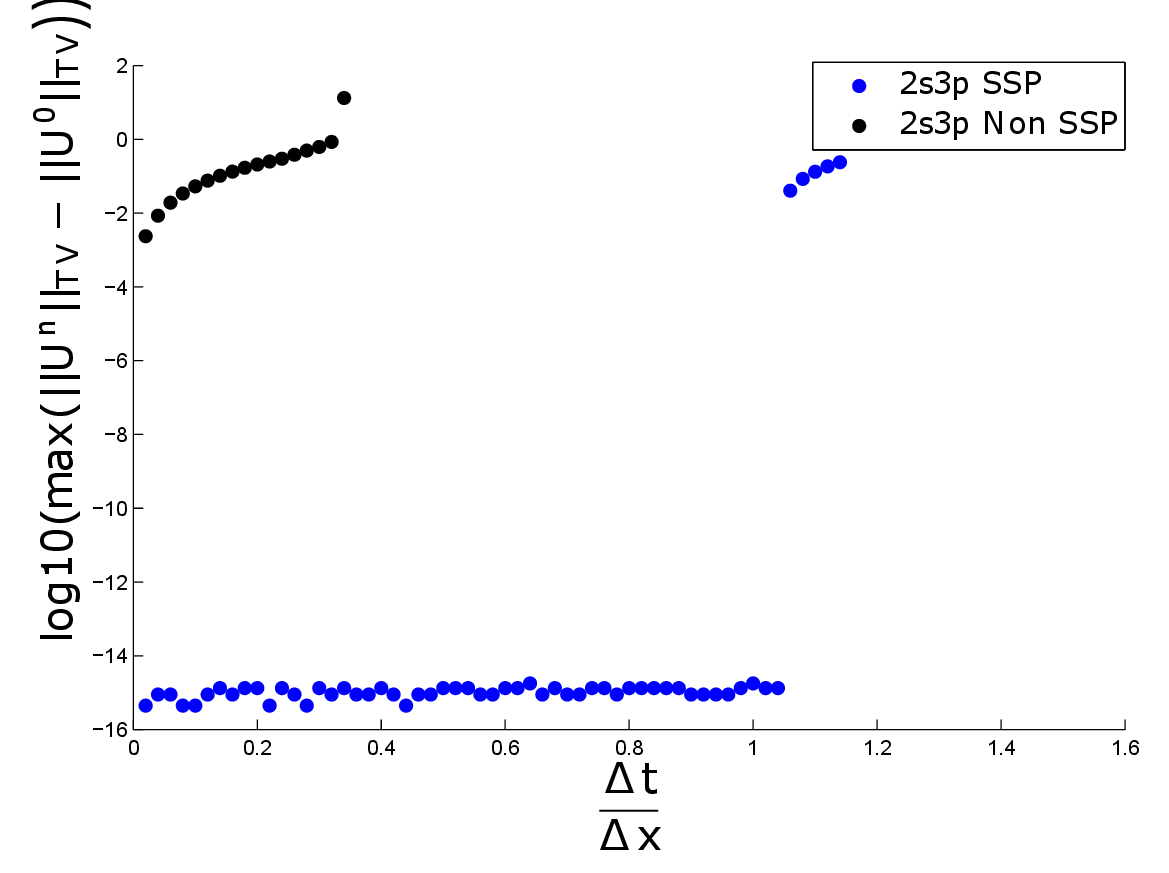}
\caption{ Comparison of the  rise in total variation as a function of the CFL number 
for the two-stage third order SSP  \eqref{2s3p} and non-SSP method \eqref{bad2s3p}.
On the left is the maximal per time-step rise  and on the right the maximal TV rise above the initial TV.
 }
 \label{LinearTest}
\end{figure}

We use  a spatial step $\Delta x = \frac{1}{1600}$, and a time-step $\Delta t = \lambda \dx$
 where we vary $0.05 \leq \lambda \leq 2$ to find the value $\lambda_{obs}$ at which the total variation starts to rise. 
 We measure the maximal rise in total variation 
 \[ \max_{0 \leq n \leq N-1} \left( \| u^{n+1} \|_{TV} - \| u^n \|_{TV} \right), \]
and the  rise in total variation compared to the total variation of the initial solution:
\[ \max_{0 \leq n \leq N-1} \left( \| u^{n+1} \|_{TV} - \| u^0 \|_{TV} \right),\]
 over time-evolution  of $N=50$ time-steps 
 using the methods in Section \ref{sec:FE-SDmethods}.
For comparison we consider the non-SSP two stage third order method, 
 \begin{eqnarray} \label{bad2s3p}
u^*     &=& u^n - \Delta t F(u^n)+ \frac{1}{2} \Delta t^2 \dot{F}(u^n), \nonumber \\
u^{n+1} &=& u^n  - \frac{1}{3}  \Delta t  F(u^n) +  \frac{4}{3} \Delta t  F(u^*) +
\frac{4}{3}    \Delta t^2  \dot{F}(u^n) +\frac{1}{2} \Delta t^2  \dot{F}(u^*).
\end{eqnarray}
In Figure \ref{LinearTest} we see that  that the SSP method \eqref{2s3p}  preserves the TVD behavior of the discretization up to the 
value $\lambda_{obs} = 1.04$, while the non-SSP method \eqref{bad2s3p} does not preserve the TVD behavior at all.
This graph shows that the  absence of the SSP property results in the loss of the TVD property for any time-step.

Figure \ref{LinearTest2} shows the maximal rise in total variation for each CFL value $\lambda = \frac{\dt}{\dx}$
for the Taylor series method, and the \eqref{2s3p}, \eqref{2s4p}, \eqref{3s5p} methods described in Section \ref{sec:FE-SDmethods}, 
as well as a three stage fourth order optimal for $K = 1/\sqrt{2}$ presented in \cite{SDpaper}. 
It is interesting to compare the observed value  $\lambda_{obs}$ at which
the TV starts to rise to the theoretical value $\sspcoef$. 
For the Taylor series method,  the two stage third order method, and the three stage fourth order method, 
the observed SSP coefficient matches exactly the theoretical value. On the other hand,
the two-stage fourth order and the three-stage fifth order, both of which have the smallest SSP coefficients (both in theory 
and practice), have a larger observed SSP coefficient than predicted, as 
presented in the following table. \\
\begin{center}
\begin{tabular}{|c|c|c|c|c|c|}  \hline 
Method & 1s2p (TS) & 2s3p & 2s4p & 3s4p & 3s5p \\ \hline
$\sspcoef$ 	  & 0.6180 & 1.0400 & 0.6788 & 1.3927  & 0.6746 \\
$\lambda_{obs}$ & 0.6180 & 1.0400 & 0.7320 & 1.3927 & 0.7136  \\ \hline
\end{tabular} 
\end{center}

\begin{figure}[t] \hspace{0.25in} 
\includegraphics[width=0.45\textwidth]{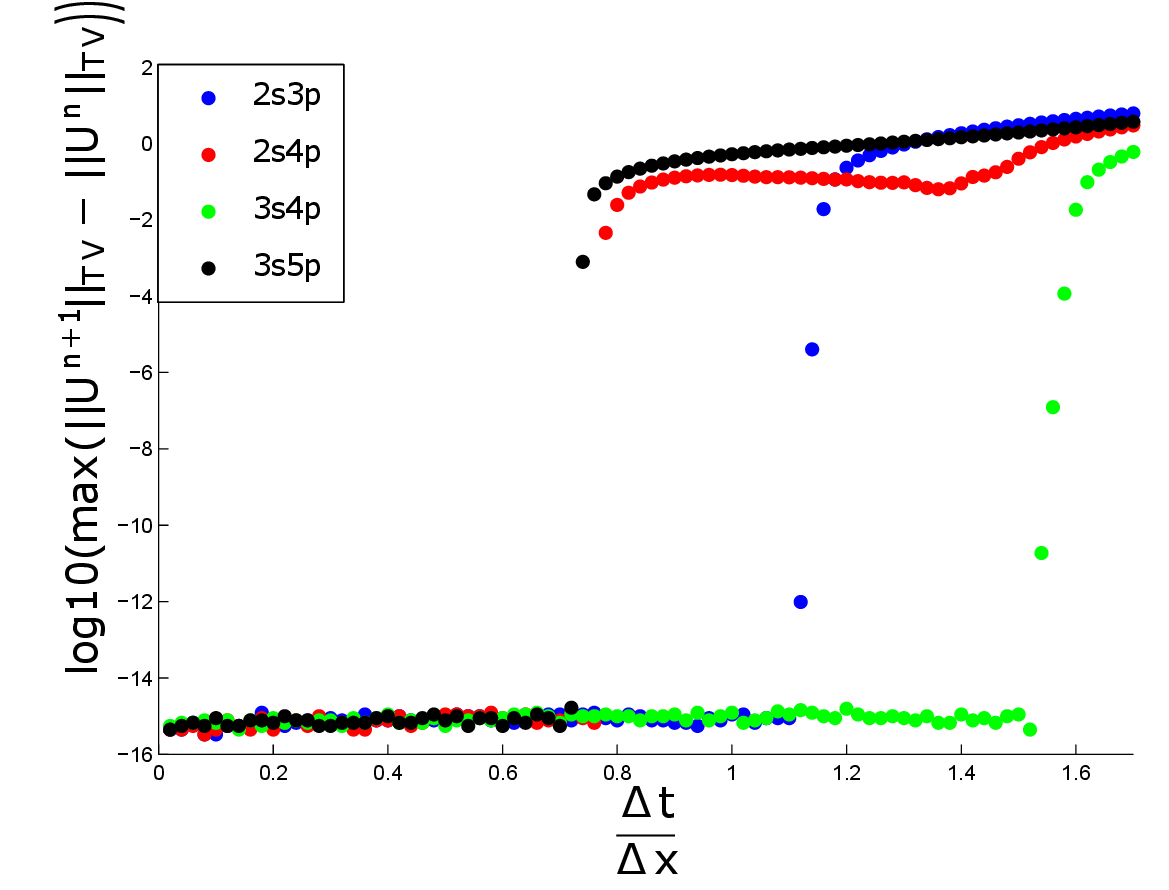}
\includegraphics[width=0.45\textwidth]{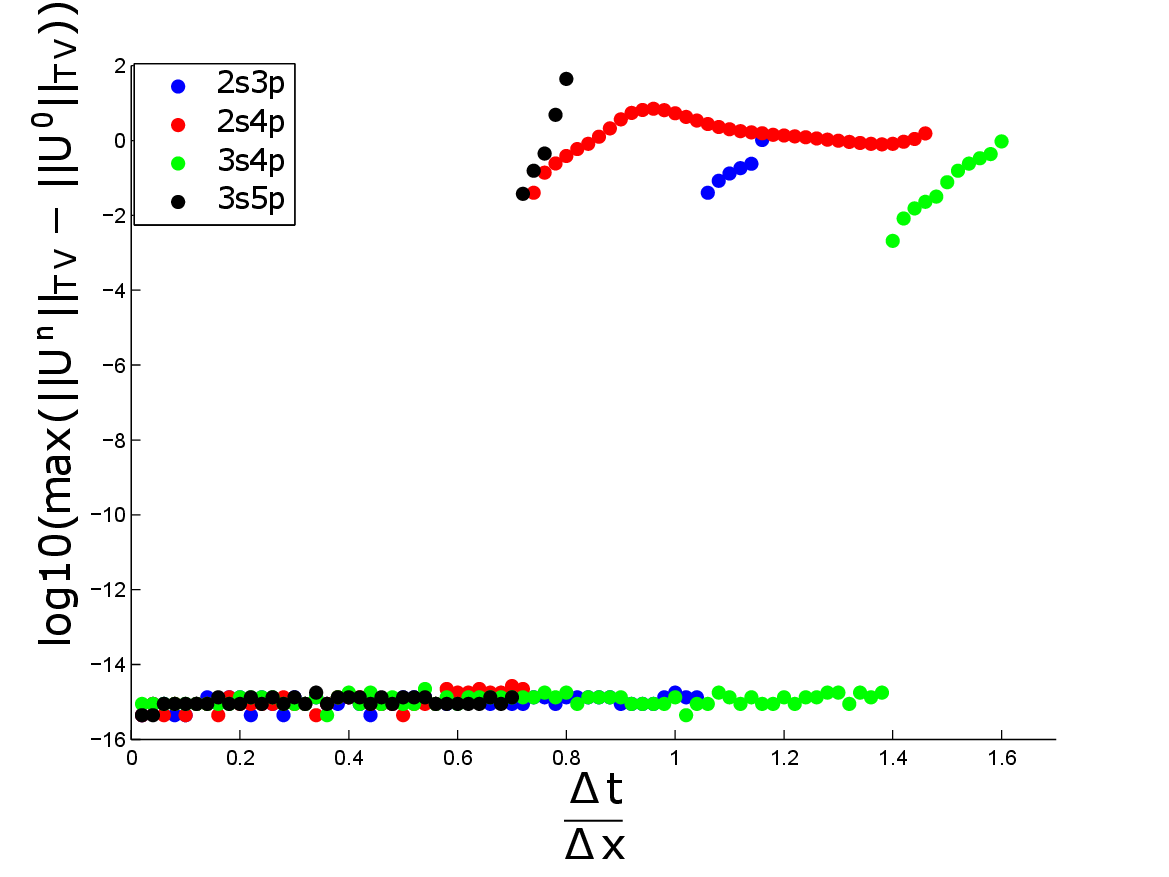} \\
\caption{Comparison of the  rise in total variation as a function of the CFL number 
for the two-stage third order SSP  \eqref{2s3p} and non-SSP method \eqref{bad2s3p}.
On the left is the maximal per time-step rise  and on the right the maximal TV rise above the initial TV.}
 \label{LinearTest2}
\end{figure}

\bigskip

\subsection{Taylor Series Condition}  \label{sec:TS}
%

In addition to being the unique one stage second order method, 
the Taylor series method \eqref{TSmethod} is a natural building block for multistage two-derivative methods.
Indeed, many spatial discretizations in the literature \cite{DumbserBalsara1,QiuDumbserShu,DaruTenaud,sealMSMD2014}
were designed to satisfy
\begin{eqnarray}  \label{TScond}
&& \mbox{\bf Taylor series condition:} \nonumber \\ 
&& \|u^n +  \Delta t F(u^n) +  \frac{1}{2} \Delta t^2 \dot{F}(u^n) \| \leq \| u^n  \|  \;\;   \mbox{for}  \;\;    \Delta t \leq \kappa \Delta t_{FE}.
\end{eqnarray}
Additionally,  \eqref{TScond} provides more flexibility for the spatial discretization: 
there exist spatial discretizations that satisfy both the  forward Euler condition \eqref{FE} 
and the Taylor series condition  \eqref{TScond}  but not  the  second derivative condition \eqref{SDcond}.
For such discretizations, the methods presented in Section \ref{sec:SD} 
are not guaranteed to preserve the desired  strong stability properties.

\begin{exm} {\bf Motivating Example:}
A simple case in which this added flexibility in the spatial discretization 
is needed is, once again, seen in the   one-way wave equation 
\[ U_t = U_x  \]
 where $F$ is defined  as before by the first-order upwind method 
\[ F(u^n)_j := \frac{1}{\dx} \left(u^n_{j+1} - u^n_j \right) \approx U_x( x_j )\]
and  $\dot{F}$ is computed by simply applying this  differentiation operator twice  
\[   \dot{F}  := \frac{1}{\dx^2} \left( u^n_{j+2}- 2 u^n_{j+1} + u^n_{j} \right).\] 
As noted above, the spatial discretization $F$ satisfies the  total variation diminishing (TVD) property:
\[ \left\| u^n + \Delta t F(u^n) \right\|_{TV} \leq   \left\| u^n\right\|_{TV}  \; \;  \; \;  \mbox{for} \; \; \; \; \Delta t \leq \Delta x,  \]
while the Taylor series term using $F$ and $\dot{F}$ satisfies 
\[ \| u^n + \Delta t F(u^n) + \frac{1}{2}  \Delta t^2 \dot{F}(u^n)  \|_{TV} \leq \|u^n \|_{TV} \; \;  \; \;  \mbox{for} \; \; \; \; \Delta t \leq \Delta x. \]
In other words, these spatial discretizations satisfy  \eqref{FEcond} and \eqref{TScond} with $\kappa=1$, in the TV  semi-norm. 
Note, however, that Condition \eqref{SDcond} is not satisfied by  this $\dot{F}$, so that   the methods derived in \cite{SDpaper} 
will  not guarantee the preservation of the TVD property of the numerical solution using $\F$ and $\dot{F}$.
\end{exm}

Once again, we wish to determine which methods can be written as convex combination of 
the base conditions, and to  find the value of $\dt$ for which we can ensure the method satisfies the
desired strong stability properties.  To decompose the schemes \eqref{MSMDmatrix}
into forward Euler and Taylor series building blocks, we stack the stages into a vector $Y$ and 
\begin{eqnarray*}
Y & = &  u^n +  \dt \mS F(Y) + \dt^2 \dot{\mS} \dot{F}(Y)  \\
& = &  \mR \ve u^n + \mP \left(  Y +  \frac{\dt}{r} F(Y) \right) +
\mQ \left(  Y +  \frac{\kappa \dt }{r} F(Y)
+ \frac{\kappa^2  \dt^2  }{2 r^2}\dot{F}(Y)  \right) .
\end{eqnarray*}
Where $R = I - \mP - \mQ$. We can easily see that if all the elements of $\mP$, $\mQ$, and $\mR \ve$ are non-negative, 
this methods will be a convex combination of the two base conditions. The following theorem gives the conditions
under which these coefficients are non-negative.

\begin{thm} \label{thmTS} 
Given spatial discretizations $F$ and $\dot{F}$ that satisfy \eqref{FE} and \eqref{TScond},
a two-derivative multistage method  of the form \eqref{MSMDmatrix} preserves  the strong stability property 
$ \| u^{n+1} \| \leq \|u^n \|$  under the time-step restriction $\Delta t \leq r \DtFE$ if it satisfies the 
component-wise
conditions
\begin{subequations} 
\begin{align} 
 \left( I +r S  +  \frac{2 r^2}{\kappa^2} \left( 1 - \kappa \right)  \dot{S} \right)^{-1} \ve  \geq  0   \label{SSPconditionTS1} \\
r  \left( I +r S  +  \frac{2 r^2}{\kappa^2} \left( 1 - \kappa \right)  \dot{S} \right)^{-1}      \left(S-  \frac{2 r}{\kappa} \dot{S}  \right)
 \geq  0   \label{SSPconditionTS2} \\
\frac{2 r^2}{\kappa^2}  \left( I +r S  +  \frac{2 r^2}{\kappa^2} \left( 1 - \kappa \right)  \dot{S} \right)^{-1}   \dot{S} 
\geq 0   \label{SSPconditionTS3}
 \end{align} 
\end{subequations}
for some $r>0$. 
\end{thm}

We showed above that the Taylor series building block is a convex combination of 
the forward Euler and the second derivative  building blocks.
Thus, any method of the form \eqref{MSMD} that can be written as a convex combination 
of the forward Euler and  Taylor series methods
can also be written as a convex combination of the forward Euler and the second derivative building blocks.
 However, the converse is not true.
 There exist certain time discretizations, such as 
 the two-stage fourth order method \eqref{2s4p}, that cannot be written as a convex combination of
forward Euler and  Taylor series methods, and therefore will not satisfy Theorem  \ref{thmTS}.
  

\subsubsection{Optimal methods and order barriers} \label{OrderBarriers}

The family of  three stage fourth order methods for  $\kappa \geq1$
is given by 
\begin{eqnarray*}
y^{(1)} &= &u^n\\ \vspace{.25 cm}
y^{(2)} &= &u^n +\Delta t F(y^{(1)}) + \frac{1}{2} \Delta t^2 \dot{F}(y^{(1)})\\
y^{(3)} &=& u^n + \frac{1}{27}  \Delta t \left(14F(y^{(1)})  + 4F(y^{(2)})\right) +\frac{2}{27}\Delta t^2 \dot{F}(y^{(1)})\\
u^{n+1}&= & u^n + \frac{1}{48}  \Delta t \left(17 F(y^{(1)})  + 4F(y^{(2)}) +27F(y^{(3)})\right)+\frac{1}{24} \Delta t^2 \dot{F}(y^{(1)}).
\end{eqnarray*}
and has $\sspcoef = 1$.

If $ \kappa \leq1$ the optimal family of three stage fourth order methods has SSP coefficient $\sspcoef = \frac{2 \kappa }{\kappa+1}$.
This method is of the form
\begin{eqnarray*}
y^{(1)} &= & u^n\\ \vspace{.25 cm}
y^{(2)} &= & u^n + a_{21} \Delta t F(y^{(1)}) + \dot{a}_{21}  \Delta t^2 \dot{F}(y^{(1)})\\
y^{(3)} &= &u^n + a_{31} \Delta t F(y^{(1)})  +  a_{32} \Delta t F(y^{(2)})  + \dot{a}_{31} \Delta t^2 \dot{F}(y^{(1)})\\
u^{n+1}&=  &u^n + b_1  \Delta t  F(y^{(1)})  + b_2  \Delta t F(y^{(2)}) + b_3  \Delta t F(y^{(3)}) +
\dot{b}_1 \Delta t^2 \dot{F}(y^{(1)}).
\end{eqnarray*}
where the  coefficients  depend on $\kappa$:\\
{\small
\noindent \begin{tabular}{lll}
$a_{21} = \frac{\kappa+1}{2} $ 
& $ b_1 =  \frac{3\kappa^5 - 9\kappa^4 - 22\kappa^3 + 30\kappa^2 + 21\kappa + 11}{ 3(\kappa - 3)^2(\kappa + 1)^3 } $ 
&$  \dot{a}_{21}=\frac{(\kappa + 1)^2}{8} $ \\
$a_{31} = \frac{(\kappa + 1)(- \kappa^3 - 2\kappa^2 + 14\kappa + 3)}{2(\kappa + 2)^3}$ 
& $b_2 = \frac{2\kappa}{3(\kappa + 1)^3} $ 
& $\dot{a}_{31}= \frac{\kappa(- \kappa^2 + 2\kappa + 3)^2}{8(\kappa + 2)^3} $ \\
$a_{32} = \frac{(\kappa + 1)(\kappa - 3)^2}{2(\kappa+2)^3} $ &
$b_3 =  \frac{2(\kappa + 2)^3}{3(\kappa - 3)^2(\kappa + 1)^3}  $ &
$ \dot{b}_1 =  -\frac{- 3\kappa^3 + 3\kappa^2 + \kappa + 1}{6(\kappa - 3)(\kappa + 1)^2 }$  .\\
\end{tabular}}

\bigskip

More optimal SSP methods that use the base  conditions \eqref{FE} and \eqref{TScond},
were found in \cite{TSpaper} up to order $p = 6$. Most of these methods cannot be written in closed form so
we do not replicate them here.  Methods of this type have an order barrier of $p \leq 6$:
\begin{thm} \cite{TSpaper}
A method of the form \eqref{MSMD} which can be decomposed into a convex combination of
the  base conditions \eqref{FE} and \eqref{TScond}, cannot have order $p \geq 7 $. 
\end{thm}

 \subsection{Negative derivative condition}  \label{sec:ND}
The second derivative and Taylor series building blocks described in Sections \ref{sec:SD} and \ref{sec:TS} 
allow for the design of explicit SSP methods that have higher order and larger SSP coefficient than
classical Runge--Kutta methods. However, similar time-step restrictions still hold for implicit two-derivative 
Runge--Kutta methods. This  constraint is due to the non-negativity of the coefficients implied by the 
form of the building blocks. Recall that unconditional implicit SSP Runge--Kutta methods were enabled by
including a downwind operator that allows for negative coefficients. Here we consider a negative coefficient on
the derivative.

Consider the  implicit Taylor series method for the  ODE \eqref{ode}
\begin{eqnarray}
 u^{n+1} & = & u^n + \dt F(u^{n+1}) - \frac{1}{2} \dt^2 \dot{F}(u^{n+1}) .
 \end{eqnarray}
This one stage method is second order and the second derivative enters with a negative sign.
Similarly, in this section we will use an implicit  negative derivative condition that states that 
the implicit building block $u^{n+1} = u^n -  \dt^2 \dot{F}(u^{n+1})$ satisfies a strong stability condition of the form
$ \|u^{n+1} \| \leq \| u^n  \| \; \; \; \forall   \Delta t > 0.$
 \begin{eqnarray} \label{iNegative}
&& \mbox{\bf Negative derivative condition}   \nonumber \\ 
&& u^{n+1} = u^n -  \dt^2 \dot{F}(u^{n+1})   \; \; \;   \implies \; \; \;
 \|u^{n+1} \| \leq \| u^n  \| 
 \; \; \; \forall   \Delta t > 0.
\end{eqnarray}

\begin{rmk}
The more natural extension of the second derivative condition \eqref{SDcond} is the 
explicit negative derivative condition
\[ \|u^n  - \Delta t^2 \dot{F}(u^n) \| \leq \| u^n  \|   \; \; \; \mbox{for}  \; \;   \Delta t \leq  \DtND. \]
This condition is stricter than \eqref{iNegative}. In fact, we can show that this explicit 
negative derivative condition implies  \eqref{iNegative}:
 \begin{eqnarray}
 u^{n+1} &=& u^n  - \Delta t^2 \dot{F}(u^{n+1})  \\
 (1+r)  u^{n+1} &=& u^n + r \left( u^{n+1}  - \frac{1}{r} \Delta t^2 \dot{F}(u^{n+1}) \right)  
 \end{eqnarray}
 Taking the norm on both sides,
  \begin{eqnarray}
 \|   u^{n+1} \| &=& \left\| \frac{1}{r+1} u^n + \frac{r}{r+1}  \left( u^{n+1}  - \frac{1}{r} \Delta t^2 \dot{F}(u^{n+1}) \right) \right\| \\ 
 &\leq &  \frac{1}{r+1} \| u^n \| + \frac{r}{r+1} \left\| \left( u^{n+1}  - \frac{1}{r} \Delta t^2 \dot{F}(u^{n+1}) \right) \right\| \\ 
 &\leq &  \frac{1}{r+1} \| u^n \| +  \frac{r}{r+1} \| u^{n+1} \|  \; \; \; \mbox{for} \; \; \; \;  \dt \leq r \DtND.
 \end{eqnarray}
 Rearranging, we obtain
 \[  \|   u^{n+1} \| \leq  \| u^n \| \; \; \; \; \; \mbox{for} \; \; \; \;  \dt \leq r \DtND.\]
 We can select $r$ to be arbitrarily large, and so $ \|   u^{n+1} \| \leq  \| u^n \|$ unconditionally.
\end{rmk}

Similarly,  it is well known \cite{hundsdorfer2003, higueras2004a,SSPbook2011} that a method of the form
 \[ u^{n+1}  =  u^n + \dt F(u^{n+1}) \] 
 is unconditionally SSP if the operator $F$ satisfies a forward Euler condition of the form \eqref{FEcond} 
 for any $\DtFE > 0$. The proof of this follows the same process described in the remark. In parallel
 to \eqref{iNegative}, we will use the less stringent condition 
 \begin{eqnarray} \label{BEcond}
&& \mbox{\bf Backward Euler  condition}   \nonumber \\ 
&& u^{n+1} = u^n +  \dt {F}(u^{n+1})   \; \; \;   \implies \; \; \;
 \|u^{n+1} \| \leq \| u^n  \| 
 \; \; \; \forall   \Delta t > 0.
\end{eqnarray}

We  wish to decompose the method into building blocks of the forms \eqref{BEcond} and 
\eqref{iNegative}.  In this case, we need the function and its derivative to {\em only} appear implicitly.
However, previous intermediate stages that use the function and its derivative may be re-used at later stages.
Thus, we use the more limited form of \eqref{MSMD} given by 
\begin{subequations} \label{iMDRK}
\begin{eqnarray}
&&u^{(i)}  =  r_i u^n + \sum_{j=1}^{i-1} p_{ij} u^{(j)}   +
\dt d_{ii} G(u^{(i)}) +   \dt^2 \dot{d}_{ii} \dot{G}(u^{(i)}), \quad  \; \; i=1, . . . , s,   \\
&&u^{n+1} =  u^{(s)}.
 \end{eqnarray}
 \end{subequations}
This form  ensures that any explicit terms in the  method \eqref{iMDRK} 
enter only after they were introduced implicitly in a prior stage.

We can write \eqref{iMDRK} in matrix form:
 \begin{eqnarray} \label{iRKmatrix}
 U = \mR \ve u^n   +  \mP U + \dt \mD G(U) +  \dt^2 \mDdot  \dot{G}(U),
 \end{eqnarray}
 where  $ \mP$ and $\mR = I - \mP $ are $s \times s$ matrices, $r_i$ are the $i$th row sum of $\mR$,  
 and $\mD$ and $\mDdot$ are $s \times s$ diagonal matrices. 
 The numerical solution $u^{n+1} $ is then given by the final element of the vector $U$.
 
 It is clear from the structure of \eqref{iMDRK} that as long as the coefficients $r_i$, $p_{ij}$, and $d_{ii}$ are non-negative, 
 and $\dot{d}_{ii}$ are non-positive, then conditions \eqref{FEcond} and \eqref{iNegative}
  ensure that the method will be unconditionally SSP. The following
 theorem formalizes this observation:
\begin{thm}  \label{thm:SSPMDRK}  \cite{Jingwei2022}
Let the operators $F$ and $\dot{F} $  satisfy 
 the forward Euler condition \eqref{FEcond}
 and the implicit negative derivative condition \eqref{iNegative}.
A method given by  \eqref{iRKmatrix} which  satisfies the  conditions
\begin{eqnarray} \label{coef}
\mR \ve  \geq  0, \; \; \; \;  \mP  \geq  0, \; \; \;  \mD \geq  0, \; \; \; \; \dot{\mD} \leq  0, 
 \end{eqnarray}
(where the inequalities are understood componentwise),
will preserve the strong stability property
\[  \|u^{n+1} \| \leq \| u^n \|  \]
for any positive time-step $\dt >0$.
\end{thm}

\begin{rmk} In \cite{Jingwei2022} we showed that the order conditions 
on a method of the form \eqref{iMDRK}   lead to negative coefficients for any method of order
$p \geq 2$. On the other hand, the forward Euler condition \eqref{FEcond} coupled with the
second derivative condition \eqref{SDcond} or the Taylor series condition \eqref{TScond}
require positive coefficients on both the function and its derivative.
Thus, implicit multi-derivative Runge--Kutta methods cannot be unconditionally SSP in the sense of
preserving the forward Euler and one of the derivative conditions  \eqref{SDcond} or \eqref{TScond} above. 
This leads us to consider the backward derivative condition if we want an unconditional SSP method.
\end{rmk}

\subsubsection{Unconditionally SSP implicit two-derivative Runge--Kutta methods up to order $p=4$} 
In this section we present the methods of orders $p=2,3,4$  that satisfy the conditions  in 
Theorem \ref{thm:SSPMDRK}. These unconditionally SSP methods were found in \cite{Jingwei2022}.

\noindent{\bf Second order:} The one-stage, second order 
method is simply the implicit Taylor series method
\[ u^{n+1} = u^n + \dt F(u^{n+1}) - \frac{1}{2} \dt^2 \dot{F}(u^{n+1}) .\]

\noindent{\bf Third order:}
A two-stage, third order unconditionally SSP implicit two-derivative Runge--Kutta method
is given by
\begin{eqnarray*}
u^{(1)}  &=&  u^n -  \frac{1}{6} \dt^2 \dot{F}(u^{(1})  \\
u^{(2)}    &=&  u^{(1)} + \dt F(u^{(2)})  -  \frac{1}{3} \dt^2\dot{F}(u^{(2)}) \\
u^{n+1} &=&   u^{(2)}.
 \end{eqnarray*}

\noindent{\bf Fourth order:} A five-stage, fourth order unconditionally 
SSP implicit two-derivative Runge--Kutta method
is given by the Shu-Osher coefficients
\[ diag(\mD) = \left[  \begin{array}{c}
   0.660949255604937\\
   0.242201390400848\\
   1.137542996287740\\
   0.191388711018110\\
   0.625266691721946\\
 \end{array} \right], \; \; \;
diag(\mDdot) = \left[  \begin{array}{c}
-0.177750705279127 \\         
-0.354733903778084 \\
-0.403963513682271  \\
-0.161628266349058 \\
 -0.218859021269943 \\
 \end{array} \right],
 \]
 \[P=\left[  \begin{array}{ccccc}
   0                   		 & 0                  & 0         &          0      &     \hspace{.2in}        0 \\
   1                   		 & 0                  &  0         &          0     &         \hspace{.2in}         0 \\
   0.084036809261019 &  0.915963190738981     &              0         &          0     &       \hspace{.2in}           0 \\
   0.001511648458457  &    0  & 0.090254853867587  &                 0          &     \hspace{.2in}        0 \\
   0          &       0 &  0 & 1 &  \hspace{.2in}      0 \\
 \end{array}  \right], \]
 \[  \mR \ve = \left[    1, 0, 0,   0.908233497673956, 0 \right]^T .\]
 
No methods of fifth order  that satisfy the conditions in Theorem  \ref{thm:SSPMDRK} were found in \cite{Jingwei2022}. 

 \section{SSP theory for Two-derivative IMEX methods}    \label{sec:2D-RK-IMEX}
The implicit negative derivative condition \eqref{iNegative}  does not 
 always offer an attractive  alternative to the second derivative and Taylor series conditions given in
\cite{SDpaper,TSpaper}, as there are some significant spatial discretizations  that do not satisfy  \eqref{iNegative}.
 However, the  implicit negative  derivative condition which enables  unconditionally SSP schemes
 is of tremendous interest in several application areas where one component of the 
problem  is stiff and which satisfies \eqref{BEcond} and  \eqref{iNegative} \cite{Jingwei2022}. 
These problems  require treatment with an implicit-explicit (IMEX) time-stepping approach, 
as we describe in this section.

 In this section we consider equations that have two components
 \begin{eqnarray} \label{IMEXeqn} 
u_t = \F(u) + \G(u).
\end{eqnarray}
In many cases,   the time-step restriction coming from the explicit evolution of 
one component ($\F$) is of a reasonable size under the forward Euler condition \eqref{FEcond}
while the second component ($\G$) imposes a very small time-step restriction when handled explicitly,
but satisfied unconditionally conditions \eqref{BEcond} and has a derivative $\Gdot$ that satisfies \eqref{iNegative}.
To alleviate this time-step restriction we can turn to IMEX two-derivative methods.
It seems natural to use  unconditionally SSP implicit multi-derivative methods for $\G$
 in combination with explicit methods for the non-stiff term $\F$.

\begin{exm} 
Our motivating example is  the Bhatnagar-Gross-Krook (BGK) equation \cite{BGK54}, 
 a widely used kinetic model introduced to mimic the full Boltzmann equation. The BGK model is given by 
\begin{equation} \label{BGK}
\partial_t U +v\cdot \nabla_x U=\frac{1}{\varepsilon} (M-U), \quad x,v \in \mathbb{R}^d,
\end{equation}
where $U=U(t,x,v)$ is the probability density function and $M$ is the {\em Maxwellian} given by
\begin{equation}
M(t,x,v)=\frac{\rho(t,x)}{(2\pi T(t,x))^{d/2}}\exp \left(-\frac{|v-w(t,x)|^2}{2T(t,x)}\right),
\end{equation}
where the density $\rho$, bulk velocity $w$ and temperature $T$ are given by the moments of $U$:
\[  \rho=\int_{\mathbb{R}^d} U \rd{v}, \quad \rho w=\int_{\mathbb{R}^d} U v \rd{v}, 
\quad  \frac{1}{2} \rho dT=\frac{1}{2}\int_{\mathbb{R}^d} U |v-w|^2 \rd{v}. 
\] 
Discretizing this equation in space so $u \approx U$, we let
 \[ \F = -  v\cdot \nabla_x u \; \; \; \mbox{and} \; \; \; \G = \frac{1}{\varepsilon} (M-u),\]
we observe that $\F$ is a typical advection term which can be discretized to satisfy a condition of the form \eqref{FEcond},
while  $\G$ is well-behaved under an implicit evolution of the form  \eqref{BEcond}.
Furthermore, it can be easily verified that 
\begin{equation}
{\Gdot}(u) = -\G(u).
\end{equation}
\end{exm}

Assume that $\F$ satisfies a forward Euler condition of the form \eqref{FEcond}
under the restriction $\dt \leq \DtFE$, 
while $\G$ satisfies a backward Euler condition of the form \eqref{BEcond}.
Furthermore, $\Gdot (u)$ satisfies an implicit negative derivative condition
of the form \eqref{iNegative}.  Theorem \eqref{thm:SSP} shows that under certain 
conditions  a  multi-derivative IMEX method 
\begin{subequations} \label{IMEX-RK}
\begin{eqnarray}
&&u^{(i)}  =  r_i u^n + \sum_{j=1}^{i-1} p_{ij} u^{(j)}    + \sum_{j=1}^{i-1} w_{ij} \left( u^{(j)}   + \frac{\dt}{r} \F(u^{(j)}  ) \right)  
\\
&& \hspace{0.8in}  + \ \dt d_{ii} \G(u^{(i)}) +  \dt^2 \dot{d}_{ii} \Gdot(u^{(i)}), \quad  \; \; i=1, . . . , s,   \nonumber \\
&&u^{n+1} =  u^{(s)} \nonumber .
 \end{eqnarray}
 \end{subequations} 
 will be SSP under a  time-step restriction of the form $\dt \leq \sspcoef \DtFE$.

Note that as before it is convenient to write the method in its matrix form:
\begin{eqnarray} \label{IMEX-RKmatrix}
 U = \mR  \ve u^n + \mP U  +  \mW \left( U + \frac{\Delta t}{r} \F(U) \right)  +  \dt \mD \G(U) +  \dt^2 \mDdot  \Gdot(U),
 \end{eqnarray}
 where  $ \mP$, $\mW$, and $\mR = I - \mP - \mW$ are $s \times s$ matrices, $r_i$ are the $i$th row sum of $\mR$,  
 $\mD$ and $\mDdot$ are $s \times s$ diagonal matrices, and $\ve$ is a vector of ones. 
 The numerical solution $u^{n+1} $ is then given by the final element of the vector $U$.

The following theorem expresses the conditions under which a method of this form is SSP with a time step 
restriction related only to $\F$:
\begin{thm} \cite{Jingwei2022} \label{thm:SSP} 
Given an operator $\F$ that satisfied condition \eqref{FEcond} with $\DtFE$,
and operators   $\G$  and $\Gdot$ that unconditionally satisfy \eqref{BEcond} and \eqref{iNegative} respectively,
  if the method given by  \eqref{IMEX-RKmatrix} with $r>0$ satisfies the  component-wise conditions
\begin{eqnarray} \label{coefIMEX}
\mR \ve  \geq  0, \; \; \; \;  \mP  \geq  0, \; \; \; \;  \mW  \geq  0, \; \; \;  \mD \geq  0, \; \; \; \; \dot{\mD} \leq  0, 
 \end{eqnarray}
then it preserves the strong stability property
\[  \|u^{n+1} \| \leq \| u^n \|  \]
under the time-step condition \[  \dt \leq r \DtFE. \]
\end{thm}
In the following section we present a second and a third order method that satisfy the requirements of Theorem \eqref{thm:SSP}.

\subsection{SSP IMEX multi-derivative Runge--Kutta methods}
 \label{sec:2D-RK-IMEXmethods}
The methods in this section were found in \cite{Jingwei2022}. They satisfy the form \eqref{IMEX-RK} and the additional 
condition $d_{ii} + \dot{d}_{ii} > 0$ for each stage $i$, which together ensure the asymptotic preserving property.
Given a function $\F$ that satisfies the forward Euler condition \eqref{FEcond},
 $\G$  that satisfies the backward Euler \eqref{BEcond}, and $\Gdot$ that satisfies an implicit negative derivative condition
\eqref{iNegative}, the following IMEX methods 
have an explicit part that is SSP for a time-step that depends only on $\F$,
and an implicit part that is unconditionally SSP. 

\smallskip

\noindent{\bf Second order method:} The method
\begin{eqnarray} \label{IMEX2d3s2p}
u^{(1)} &  = &  u^n   + \frac{1}{2} \dt  \G(u^{(1)})  \nonumber \\
u^{(2)} &  = &   u^{(1)}   + \dt \F(u^{(1)} )  
 -   \frac{1}{2}  \dt^2  \Gdot(u^{(2)})  \\
u^{(3)} &  = &  \frac{1}{2} u^{(1)}    +   \frac{1}{2}  \left( u^{(2)}   + \dt \F(u^{(2)}  ) \right)  
+  \frac{1}{2} \dt \G(u^{(3)})  \nonumber  \\
  u^{n+1} &=&  u^{(3)}. \nonumber
 \end{eqnarray}
is  SSP  for $\dt \leq \DtFE$ arising from condition  \eqref{FEcond} satisfied by $\F$.
The implicit component does not result in any restriction on the allowable time-step.

 \smallskip

\noindent{\bf Third order method:} The six stage method given by
\begin{eqnarray}
u^{(1)} &  = &  r_1 u^n   + \dot{d}_{11}  \frac{1}{2} \dt  \G(u^{(1)})  \nonumber\\
u^{(2)} &  = &  r_2  u^n   + p_{21} u^{(1)}   + w_{21} \left( u^{(1)}   + \frac{\dt}{r} \F(u^{(1)})  \right) \nonumber \\
&& + \ d_{22} \dt \G(u^{(2)})  + \dot{d}_{22} \dt^2 \Gdot(u^{(2)})  \nonumber \\
u^{(3)} &  = &    p_{32} u^{(2)}   + w_{31} \left( u^{(1)}   + \frac{\dt}{r} \F(u^{(1)} ) \right)   + \ d_{33} \dt \G(u^{(3)})    \\
u^{(4)} &  = &  r_4  u^n   + p_{42} u^{(2)}   + w_{43} \left( u^{(3)}   + \frac{\dt}{r} \F(u^{(3)})  \right)   + d_{44} \dt \G(u^{(4)})\nonumber   \\
u^{(5)} &  = &   p_{51} u^{(1)}   +  p_{52} u^{(2)}   + w_{51} \left( u^{(1)}   + \frac{\dt}{r} \F(u^{(1)})  \right) \nonumber \\
&& + w_{54} \left( u^{(4)}   + \frac{\dt}{r} \F(u^{(4)})  \right)   + d_{55} \dt \G(u^{(5)})  + \dot{d}_{55} \dt^2 \Gdot(u^{(5)}) \nonumber  \\
u^{(6)} &  = &   p_{61} u^{(1)}   +  p_{65} u^{(5)}   + w_{61} \left( u^{(1)}   + \frac{\dt}{r} \F(u^{(1)})  \right)  \nonumber \\
&& + w_{64} \left( u^{(4)}   + \frac{\dt}{r} \F(u^{(4)} )  \right) + \dot{d}_{66} \dt^2 \Gdot(u^{(6)}) \nonumber  \\
  u^{n+1} &=&  u^{(6)} \nonumber
 \end{eqnarray}
 where
 \[ r_1 = 1, \; \; \;  r_2 =    0.688151680893388, \; \; \;  r_4 =  0.583517183806433,\]
 \[
 \begin{array}{lll}
p_{21} = 0.253395246357353,  &  \; \; \; \; & w_{21} =0.058453072749259 \\
 p_{32} = 0.235733481708505, & & w_{31} =0.764266518291495 \\
 p_{42} = 0.123961833526104, & & w_{43} = 0.292520982667463 \\
p_{51} =    0.409037644509411, & & w_{51} = 0.173788618990251 \\
p_{52} = 0.136123556305509, & &  w_{54} =  0.281050180194829  \\
p_{61} =  0.203353399602184,  & & w_{61} =  0.016811671845949\\
 p_{65} = 0.331204417210324,  & & w_{64} = 0.448630511341543\\
 d_{22} = 2,      				& \; \; \; \; & \dot{d}_{11} =  - 0.871358934880525  \\
d_{33} = 0.388820513661584, & & \dot{d}_{22} =  - 0.856842702601821  \\
d_{44} = 0.083529464436389, & &\dot{d}_{55} =  - 2				  \\
 d_{55} = 1.793313488277995, &  & \dot{d}_{66} = - 0.205134529930013  \\
 \end{array}
\]
 is SSP  for $\dt \leq \sspcoef \DtFE$ where $\sspcoef=0.904402174130635$ and $\DtFE$
comes from condition  \eqref{FEcond} satisfied by $\F$.

These methods were tested extensively in \cite{Jingwei2022} where we showed that
for a hyperbolic relaxation model,  the Broadwell model, and the BGK kinetic equations, 
these methods provide the desired high order accuracy, positivity preservation, and
an asymptotic preserving property.

\section{Multistep multi-stage two-derivative methods}  \label{sec:2D-GLM-IMEX}
 Moradi et al.  \cite{MORADI2019}   extended the study of
 SSP two derivative Runge--Kutta methods by including previous steps.
 They developed  SSP general linear methods (GLMs) with two-derivatives by
using the second derivative condition \eqref{SDcond}. 
In this section we extend the  SSP IMEX two derivative Runge--Kutta methods \eqref{IMEX-RK} 
considered in \cite{Jingwei2022}  by including  previous  steps. 
We will study the SSP properties of these IMEX two derivative  GLMs
using the implicit negative derivative condition \eqref{iNegative}.

A $k$-step $s$-stage two-derivative IMEX method is given by
\begin{eqnarray} \label{IMEX-GLM}
y^{(i)}  &= & \sum_{\ell=1}^k r_{i\ell} u^{n+\ell - k } + \sum_{j=1}^{i-1} p_{ij} y^{(j)}    
+ \sum_{j=1}^{i-1} w_{ij} \left( y^{(j)}   + \frac{\dt}{r} \F(y^{(j)}  ) \right)  
\\
&& \hspace{0.8in}  + \ \dt d_{ii} \G(y^{(i)}) +  \dt^2 \dot{d}_{ii} \Gdot(y^{(i)}), \quad  \; \; i=1, . . . , s,   \nonumber \\
u^{n+1} &= &   \sum_{\ell=1}^k \gamma_{\ell} u^{n+\ell - k} + \sum_{j=1}^{s} q_{j} y^{(j)}    
+ \sum_{j=1}^{s} v_{j} \left( y^{(j)}   + \frac{\dt}{r} \F(y^{(j)}  ) \right)   \nonumber .
 \end{eqnarray}
This can be written in the matrix form
\begin{eqnarray} \label{IMEX-GLMmatrix}
 Y &= & \mR  U + \mP Y  +  \mW \left( Y + \frac{\Delta t}{r} \F(Y) \right)  +  \dt \mD \G(Y) +  \dt^2 \mDdot  \Gdot(Y), \\
 u^{n+1} &= &   \Gamma U + \mQ Y + \mV \left( Y + \frac{\Delta t}{r} \F(Y) \right) 
 \end{eqnarray}

\begin{rmk}
Note that the form  \eqref{IMEX-RK} is not equivalent to  \eqref{IMEX-GLM}   with $k=1$, due to the inclusion
of additional explicit terms in the final stage of  \eqref{IMEX-GLM}. Recall that in \eqref{IMEX-RK}
we required an implicit evaluation at each stage, and set $u^{n+1} =  y^{(s)}$. This was done in \cite{Jingwei2022}
to ensure the asymptotic preserving (AP) property holds. In this section we allow any stage to be explicit, and 
in particular we allow the final stage to have convex combinations of prior steps, stages, and forward Euler steps
of the explicit operator $\F$. This additional freedom allows for larger SSP coefficients, as we will observe in 
Section \ref{ND-IMEX-GLMmethods}.
\end{rmk}

\begin{thm}  \label{thm:GLM} 
Given an operator $\F$ that satisfied condition \eqref{FEcond} with $\DtFE$,
and operators   $\G$  and $\Gdot$ that unconditionally satisfy \eqref{BEcond} and \eqref{iNegative} respectively,
  if the method given by  \eqref{IMEX-GLM} with $r>0$ satisfies the  component-wise conditions
\begin{eqnarray} \label{coefIMEXGLM}
\mR   \geq  0, \; \; \; \;  \mP  \geq  0, \; \; \; \;  \mW  \geq  0, \; \; \;  \mD \geq  0, \; \; \; \; \dot{\mD} \leq  0, 
 \; \; \;   \Gamma \geq  0,  \; \; \;  \mQ \geq  0,  \; \; \;   \mV \geq  0,
 \end{eqnarray}
then it preserves the strong stability property
\[  \|u^{n+1} \| \leq \max \left\{ \| u^{n+1-k} \| , \|u^{n+2-k}  \|, . . ., \|u^{n}  \| \right\} \]
under the time-step condition \[  \dt \leq r \DtFE = \sspcoef \DtFE. \]
\end{thm}
\begin{proof}
We assume that we begin the simulation with $k$ starting values $u^0, . . ., u^{k-1}$ that are well behaved.
Observe that  by consistency the coefficients of the first stage
 $ \sum_{\ell=1}^k r_{1\ell}   =1$, and the conditions in the theorem  require $r_{1\ell} \geq 0 $ for $\ell = 1, ..., k$,
so that  
\[  \left\| \sum_{\ell=1}^k r_{1\ell} u^{n+\ell -k } \right\| \leq 
\sum_{\ell=1}^k r_{1\ell}   \left\|  u^{n+\ell -k } \right\| \leq  \max_{1 \leq \ell \leq k} \left\|  u^{n+\ell -k } \right\| .
\]
Now consider the first stage of the method 
\[
 y^{(1)}    =   \sum_{\ell=1}^k r_{1\ell} u^{n+\ell -k }  +   \dt d_{11}  \G(y^{(1)})  +  \dt^2  \dot{d}_{11}  \Gdot(y^{(1)}) .
\]
The fact that $\G$ and $\Gdot$ satisfy Condition \eqref{iNegative} allows us to conclude that to conclude that 
\[
 \left\| y^{(1)} \right\|    \leq    \left\|  \sum_{\ell=1}^k r_{1\ell} u^{n+\ell -k }  \right\| \leq  \max_{1 \leq \ell \leq k} \left\|  u^{n+\ell -k} \right\| 
 \; \; \; \; \; \forall  \; \dt .
\]

We now proceed to consider each stage in turn.
At the $i$th stage we have already shown that the previous stages 
satisfy the strong stability condition
\[  \left\| y^{(j)} \right\|  \leq  \max_{1 \leq \ell \leq k} \left\|  u^{n+ \ell -k } \right\|  , \]
for $j=1,..., i-1$.
The fact that $\F$ satisfies the forward Euler condition \eqref{FEcond} gives us 
\[  \left\| y^{(j)}   + \frac{\dt}{r} \F(y^{(j)}  ) \right\| \leq  \left\| y^{(j)} \right\|  \; \; \; \; \forall \; \;  \dt \leq r \DtFE .\]
Putting this together we observe that, due to the positivity of the coefficients the explicit terms at the $i$th term satisfy
\begin{eqnarray*}
 \left\|  y^{(i)}_{ex} \right\| &=&  \left\|  \sum_{\ell=1}^k r_{i\ell} u^{n + \ell -k} + \sum_{j=1}^{i-1} p_{ij} y^{(j)}    
+ \sum_{j=1}^{i-1} w_{ij} \left( y^{(j)}   + \frac{\dt}{r} \F(y^{(j)}   \right)  \right\|  \\
&\leq& \sum_{\ell=1}^k r_{i\ell}  \left\|   u^{n + \ell -k} \right\| + \sum_{j=1}^{i-1} p_{ij}  \left\|   y^{(j)}   \right\| 
+ \sum_{j=1}^{i-1} w_{ij}  \left\|   \left( y^{(j)}   + \frac{\dt}{r} \F(y^{(j)}   \right)  \right\| \\
&\leq& \left( \sum_{\ell=1}^k r_{i\ell}  + \sum_{j=1}^{i-1} p_{ij} 
+ \sum_{j=1}^{i-1} w_{ij}  \right)   \max_{1 \leq \ell \leq k} \left\|  u^{n+\ell -k } \right\|   \\
&\leq&  \max_{1 \leq \ell \leq k} \left\|  u^{n+\ell -k } \right\|   \\
\end{eqnarray*}
where the last equality is by the consistency conditions.
Observing that $y^{(i)} $ takes the form
\[ y^{(i)} = y^{(i)}_{ex} + \ \dt d_{ii} \G(y^{(i)}) +  \dt^2 \dot{d}_{ii} \Gdot(y^{(i)}) \]
and that $\G$ and $\Gdot$ satisfy Condition \eqref{iNegative}, we conclude that
\[ y^{(i)}   \leq  \max_{1 \leq \ell \leq k} \left\|  u^{n+ \ell -k } \right\|  .\]

Finally, the positivity of the coefficients means that the the last stage consists only of a convex combination so that 
\begin{eqnarray*}
\left\| u^{n+1} \right\| &\leq &    \sum_{\ell=1}^k \gamma_{\ell} \left\| u^{n+\ell - k} \right\|  
+ \sum_{j=1}^{s} q_{j} \left\| y^{(j)}    \right\| 
+ \sum_{j=1}^{s} v_{j} \left\| y^{(j)}   + \frac{\dt}{r} \F(y^{(j)}  ) \right\| . \\
\end{eqnarray*}
Once again, the  fact that $\F$ satisfies the forward Euler condition \eqref{FEcond} gives us 
\begin{eqnarray*}
\left\| u^{n+1} \right\| &\leq &  \left(  \sum_{\ell=1}^k \gamma_{\ell} +  \sum_{j=1}^{s} q_{j} 
+  \sum_{j=1}^{s} v_{j} \right) \max_{1 \leq \ell \leq k} \left\|  u^{n+ \ell -k } \right\|,
\end{eqnarray*}
for  $\dt \leq r \DtFE$. 
By the consistency conditions, these coefficients all add to one, so we obtain
\[ \left\| u^{n+1} \right\| \leq \max_{1 \leq \ell \leq k} \left\|  u^{n+ \ell -k } \right\|.\]
\end{proof}

\subsection{New SSP IMEX two-derivative GLM methods} 
\label{ND-IMEX-GLMmethods}

\subsubsection{Second order methods} 
In Section \ref{sec:2D-RK-IMEXmethods} we presented a one-step two-derivative  three stage second order method \eqref{IMEX2d3s2p} with $\sspcoef =1$. 
However that method did not include explicit evaluations  in the final stage because we wanted the AP property. If we allow the 
structure \eqref{IMEX-GLM} we can obtain a one-step two-derivative  three stage second order method with better SSP coefficient.

 
 \smallskip
 
\noindent{\bf One-step two-derivative  three stage second order method}
\begin{eqnarray} 
y^{(1)}  &= & u^{n} + \frac{1}{2+\sqrt{2}}   \dt  \G(y^{(1)}) - \frac{1}{2+\sqrt{2}}  \dt^2  \Gdot(y^{(1)})  \nonumber \\
y^{(2)}  &= &  \left( y^{(1)}   + \frac{\dt}{r} \F(y^{(1)}  ) \right)  \nonumber \\
y^{(3)}  &= &  \frac{6- \sqrt{2}}{8} y^{(1)}    + 
\frac{2+ \sqrt{2}}{8}  \left( y^{(2)}   + \frac{\dt}{r} \F(y^{(2)}  ) \right)   + \frac{1}{\sqrt{2}}  \dt  \G(y^{(3)})  \nonumber \\
u^{n+1} &= &   \left( 1 - \frac{\sqrt{2}}{4} \right)  y^{(3)}    
+ \frac{2+\sqrt{2}}{4(1+\sqrt{2})}   \left( y^{(3)}   + \frac{\dt}{r} \F(y^{(3)}  ) \right)  .
 \end{eqnarray}
preserves  the SSP properties of \eqref{BEcond} and \eqref{iNegative} with 
$\sspcoef = (1+\sqrt{2})/2 $.

\smallskip

\noindent{\bf Two-step two-derivative  three stage second order method} 
 \begin{eqnarray} \label{GLM2s2k2p}
y^{(1)}  &= & u^{n-1 }    \nonumber  \\
y^{(2)}  &= &  \left( y^{(1)}   + \frac{\dt}{r} \F(y^{(1)}  ) \right)   \nonumber  \\
y^{(3)}  &= & r_{31} u^{n }  +  w_{32} \left( y^{(2)}   + \frac{\dt}{r} \F(y^{(2)}  ) \right)   \nonumber  \\
&+&   \dt d_{33} \G(y^{(3)}) +  \dt^2 \dot{d}_{33} \Gdot(y^{(i)})  \nonumber \\
u^{n+1} &= & q_{3} y^{(3)}    +  v_{2} \left( y^{(2)}   + \frac{\dt}{r} \F(y^{(2)}  ) \right)   +  v_{3} \left( y^{(3)}   + \frac{\dt}{r} \F(y^{(3)}  ) \right)
 \end{eqnarray}
 with coefficients
 \[ w_{32} = \frac{ (27 + 5\sqrt{29})^{1/3} - 2^{1/3} }{2 (5 + \sqrt{29})^{1/3}} , 
\; \; \; \dot{d}_{33} = - \frac{2}{v_3 + q_3} , \]
\[d_{33} = \frac{10}{6}   +  \frac{2^{2/3}}{6}  \left( (9 \sqrt{29} + 43)^{1/3}  - (9 \sqrt{29} - 43)^{1/3}  \right)  \]
\[ r_{31} = 1 + \frac{1}{2} (\frac{1}{2}  (-5 + \sqrt{29}))^{1/3} - 1/(2^{2/3} (-5 + \sqrt{29})^{1/3}) \]
 \[ q_3 = \frac{20}{9}   -  \frac{2^{2/3}}{9} \left( (727-135 \sqrt{29})^{1/3} + (727+135 \sqrt{29})^{1/3} \right) \]
\[ v_2  = \frac{1}{9} \left( 7 -  2^{2/3}( 81\sqrt{29} + 137 )^{1/3} +  2^{2/3}( 81\sqrt{29} - 137 )^{1/3} \right),\]
\[ v_3  = 2^{2/3} \left(  ( 5 + \sqrt{29} )^{1/3} - ( \sqrt{29} -5)^{1/3} \right)  - 2,\]
 is SSP with \[ r = \sspcoef = \frac{1}{3} \left( 
 \left(\frac{1}{2}(61 + 9 \sqrt{29})\right)^{1/3} +
 \left(\frac{1}{2}(61 - 9 \sqrt{29})\right)^{1/3}  - 1\right)  \approx 1.5468.\]

\smallskip

\noindent{\bf $(k \geq 3)$-step two-derivative  two-stage second order methods:}
Increasing the number of steps allows fewer stages, which means fewer implicit
evaluations.  Here we present a new family of two-derivative $k$-step two stage second order methods 
for $k \geq 3$. These methods have
\[ r = \sspcoef = \frac{k-2}{k-1}\]
and take the form
   \begin{eqnarray} \label{GLM2sKk2p}
y^{(1)}  &= &  u^n  -  (k-1) \dt^2    \Gdot(y^{(1)})  \nonumber \\
y^{(2)}  &= & \frac{1}{k-1} u^{n-k+1}   +   \frac{k-2}{k-1}  \left(  y^{(1)} +  \frac{\dt}{r} \F(y^{(1)}  ) \right) \nonumber \\
& & \; \; \; \; \; \; \; +  k    \dt  \G(y^{(2)})  -  k \dt^2   \Gdot(y^{(2)})  \nonumber \\
u^{n+1} &= &    \frac{1}{k-1}  y^{(2)}  +   \frac{k-2}{k-1}  \left( y^{(1)}   + \frac{\dt}{r} \F(y^{(1)}  ) \right) 
 \end{eqnarray}

\subsubsection{Third order methods}

A two-step, two-derivative GLM with $s=5$ stages has $r=\sspcoef = 1.080445742835932$
 is given by the coefficients

\[ R = \left[ \begin{array}{lll}
0 & 0 & 1 \\
0.000000000013270   & 0.403826433558741   & 0.037615230472512 \\
0  &  0.221598110956903   & 0 \\
0  &  0.059380532720245  & 0  \\
0 & 0 & 0 \\
\end{array} \right].\]

\[ P = \left[ \begin{array}{lllll}
    0 & 0 & 0 & 0 & 0  \\
   0.452661697511965    & 0 & 0 & 0 & 0 \\
  0 &    0.032510664101898         &          0        &           0  & 0 \\
   0.235231740166619   & 0.000000000563127   &          0        &           0  & 0 \\
   0.536915718824635   & 0.013138165959401   &          0        &           0   \\
   \end{array} \right].\]

\[ W =  \left[ \begin{array}{lllll}
     0 & 0 & 0 & 0 & 0  \\
   0.105896638443513  & 0 & 0 & 0 & 0  \\
   0.745891224941199  & 0 & 0 & 0 & 0  \\
0 &    0 &    0.705387726550010       &            0 & 0 \\
   0.409669470298833   & 0 &   0.000000000119198 &  0.040276644797934 & 0 \\
\end{array} \right].\]

\[ \mD = diag \left[   0,
  21.332739593864588,
0,
   0.652867317315466,
  14.945015954497144 \right]. \]
\[   \mDdot = diag \left[ -6.7737812489230,   - 72.4600167654208, 0, 0,  -161.5846694139845\right]. \]

\[ \Gamma =  \left[ 0, 0,0 \right]. \]

\[ \mQ =    \left[ 0.289233938741249, 0, 0, 0,  0.041812814961867 \right]. \] 

\[ \mV = \left[  0.274172259985154,  0 , 0,  0.394780986311730, 0 \right]. \]

\section{Conclusions}

The increasing popularity of multi-stage multiderivative methods in the  time-evolution of  hyperbolic problems, 
raised interest in their strong stability properties. 
In this paper we review the recent work on SSP two-derivative methods.
We first discuss the  SSP formulation for multistage two-derivative methods presented in \cite{SDpaper}
 where we require the spatial discretization to satisfy the forward Euler condition \eqref{FEcond}
 and a second derivative condition of the form \eqref{SDcond}. 
 We present the conditions under which we can ensure that a  
 explicit SSP multistage two-derivative methods  preserve the strong stabilities properties of
 \eqref{FEcond} and \eqref{SDcond}. We present optimal methods of up to order $p=5$ and
 demonstrate the need for SSP time-stepping methods in simulations
where the spatial discretization is specially designed, as well as the 
sharpness of the SSP time-step  for some of these methods.

While this choice of base conditions  gives more flexibility in finding  SSP time stepping schemes, 
it limits the flexibility in the choice of the spatial discretization.  
For this reason, we considered in \cite{TSpaper} an alternative SSP formulation based on the  
conditions \eqref{FEcond} and \eqref{TScond} and   investigate SSP methods that preserve their  strong stability properties.  
 These base conditions are relevant in many commonly used spatial discretizations that
 are designed to satisfy the Taylor series condition \eqref{TScond} but
  may not satisfy the second derivative condition \eqref{SDcond}.
 Explicit SSP  methods which preserve the strong stability properties of  \eqref{FEcond} and \eqref{TScond}, 
have a maximum obtainable order of $p=6$, as we proved in \cite{TSpaper}.
While this approach decreases the flexibility in our choice of time discretization 
but it increases the  flexibility in our choice of spatial  discretizations.
Numerical tests presented in \cite{TSpaper} showed that this increased flexibility 
allowed for more efficient simulations in several cases.

We showed in \cite{Jingwei2022} that the $p \geq 2$  order conditions for implicit two-derivative Runge--Kutta method
 lead to negative coefficients, so that requiring that the second derivative satisfy the conditions
 \eqref{SDcond} or  \eqref{TScond} which have   positive coefficients  results in conditional SSP.
Instead, in \cite{Jingwei2022}  we presented a class of unconditionally SSP implicit multi-derivative  
Runge--Kutta schemes enabled by the backward derivative condition \eqref{iNegative} 
as an alternative to the second derivative conditions given in \cite{SDpaper,TSpaper}.
We  review this SSP theory here and reproduce  unconditionally SSP methods of order $p=2, 3, 4$  for use 
with spatial discretizations that satisfy these conditions.
This implicit negative derivative condition is extremely relevant for certain classes of problems including
 the Broadwell model and the BGK kinetic equation. In  \cite{Jingwei2022} we 
formulated  two-derivative  IMEX  Runge--Kutta  methods of order $p=2$ and $p=3$
that are SSP under a time-step restriction independent of the stiff term. These are reproduced in Section \ref{sec:2D-RK-IMEXmethods}.

Finally, we extend this SSP theory for two-derivative IMEX Runge--Kutta methods based on the derivative conditions
\eqref{BEcond} and \eqref{iNegative}  to two-derivative IMEX GLMs \ref{sec:2D-GLM-IMEX},
 and present novel methods of orders $p=2$ and $p=3$ in Section \ref{ND-IMEX-GLMmethods}.
 These methods have fewer stages and  larger SSP coefficient than the corresponding 
 SSP two-derivative IMEX Runge--Kutta methods, while still having no step-size restriction resulting from the 
 implicit component. We hope that the new approach and results for 
SSP two-derivative IMEX GLM methods will be of significant use  in the simulation of 
relevant problems.

\newpage 
\appendix

\section{Order Conditions} \label{OC-MDRK}
For a  method of the form \eqref{MSMDmatrix}  to be of order $p=P$, 
the coefficients $\mA, \mAdot, \vb, \bdot$  need to satisfy the order conditions given below.
For simplicity, we define auxiliary coefficients 
$\vc = \mA \ve$ and $\vcdot = \mAdot \ve$, where $\ve$ is a vector of ones.
\\

\noindent \begin{tabular}{ll}
 $p = 1$  & $ \vb^T e =1$ \\ 
$p = 2$  & $ \vb^T c+\bdot^Te =\frac{1}{2}$   \\
$p = 3$  &  $ \vb^T \vc^2 + 2 \bdot^T \vc=\frac{1}{3} $ \\ 
& $ \vb \mA \vc + \vb^T \vcdot + \bdot^T \vc=\frac{1}{6} $ \\
$p = 4$ & $\vb^T \vc^3 + 3 \bdot^T \vc^2 = \frac{1}{4} $   \\
& $ \vb^T \left( \vc \odot \mA \vc \right)+ \vb^T \left( \vc \odot \vcdot \right)+\bdot^T \vc^2 + \bdot^T \mA \vc + \bdot^T \vcdot = \frac{1}{8} $ \\
 & $ \vb^T \mA \vc^2+ 2 \vb^T \mAdot \vc+ \bdot^T \vc^2=\frac{1}{12} $\\
&  $ \vb^T \mA^2 \vc+ \vb^T \mA \vcdot + \vb^T \mAdot \vc + 
\bdot^T \mA \vc + \bdot^T \vcdot=\frac{1}{24} $\\
$p = 5$  &   $    \vb^T \vc^4 + 4 \bdot^T \vc^3 =\frac{1}{5} $ \\
	      &   $ \vb^T \left( \vc^2 \odot \mA \vc \right)  + \vb^T\left( \vc^2 \odot \vcdot \right)+\bdot^T \vc^3+
	      2 \bdot^T \left( \vc \odot \mA \vc \right)+ 2 \bdot^T\left( \vc \odot \vcdot \right)=\frac{1}{10} $  \\
	           &  $ \vb^T \left( \vc \odot \mA \vc^2 \right)+ 2 \vb^T \left(\vc \odot \mAdot \vc \right)
	           +\bdot^T \vc^3+ \bdot^T \mA \vc^2 + 2 \bdot^T \mAdot \vc=\frac{1}{15} $\\ 
	           &  $\vb^T \left( \vc \odot \mA^2 \vc \right)+ \vb^T \left( \vc \odot \mA \vcdot \right)+ \vb^T \left(\vc \odot \mAdot \vc \right)$ \\
	           &  \; \; \;  \; \; \; \;   $ + \bdot^T \left( \vc \odot \mA \vc \right)+\bdot^T \left(\vc \odot \vcdot \right) 
	         + \bdot^T \mA^2 \vc + \bdot^T \mA \vc + \bdot^T \mAdot \vc=\frac{1}{30}$ \\ 
	           &  $ \vb^T \left( \mA \vc \odot \mA \vc \right) + 2 \vb^T \left( \vcdot \odot \mA \vc \right)
	           + \vb^T \vcdot^2 + 2 \bdot^T \left(\vc \odot \mA \vc \right) + 2 \bdot^T \left(\vc \odot \vcdot \right)=\frac{1}{20} $  \\
	          &  $ \vb^T \mA \vc^3 + 3\vb^T \mAdot \vc^2+ \bdot^T \vc^3=\frac{1}{20}  $\\
	           & $ \vb^T \mA \left( \vc	 \odot  \mA \vc \right) + \vb^T \mA \left(\vc \odot \vcdot \right) + \vb^T \mAdot \vc^2 + \vb^T \mAdot \mA \vc
	           + \vb^T \mAdot \vcdot  $ \\
	           & \; \; \; \; \; \; \; \; \; \; \;  $ + \bdot^T \left( \vc \odot \mA \vc \right)+\bdot^T \left( \vc \odot \vcdot \right)=\frac{1}{40} $ \\
\vspace{0.075in}	& $\vb^T \mA^2 \vc^2 + 2 \vb^T \mA \mAdot \vc + \vb^T \mAdot \vc^2 + \bdot^T \mA \vc^2 + 
	2 \bdot^T \mAdot \vc =\frac{1}{60} $ \\
&	$ \vb^T \mA^3 \vc+ \vb^T \mA^2 \vcdot + \vb^T \mA \mAdot \vc + \vb^T \mAdot \mA \vc + \vb^T \mAdot \vcdot
	+ \bdot^T \mA^2 \vc $ \\
&  \; \; \; \; \; \; \; \; \; \; \; 	$+\bdot^T \mA \vcdot + \bdot^T \mAdot \vc = \frac{1}{120} $ \\
	\end{tabular}
	\\
	{
\noindent \begin{tabular}{ll}
$p = 6$     &    $  \vb^T \vc^5 +  5 \bdot^T \vc^4 =  \frac{1}{6}  $ \\
& $\vb^T \left(\vc^3 \odot \mA \vc \right) + 3 \bdot^T \left(\vc^2 \odot \mA \vc \right) + \bdot^T \vc^4 + \vb^T \left(\vc^3 \odot \vcdot \right) $ \\
&  \; \; \; \; \; \; \; \;   $+  3 \bdot^T \left(\vc^2 \odot \vcdot \right)  =  \frac{1}{12}$\\
&  $ \vb^T \left( \vc^2 \odot \mA \vc^2\right) + 2 \bdot^T\left(\vc \odot \mA \vc^2 \right) + 2 \vb^T \left(\vc^2 \odot \mAdot \vc \right) 
 +  \bdot^T \vc^4 $ \\
 & \; \; \; \; \; \; \; \;  $+ 4 \bdot^T \left( \vc \odot \mAdot \vc \right)  =  \frac{1}{18} $ \\
&  $ \vb^T \left( \vc \odot \mA \vc^3\right) + 3 \vb^T \left( \vc \odot \mAdot \vc^2 \right) + 
\bdot^T \mA \vc^3 +  3 \bdot^T \mAdot \vc^2 + \bdot^T \vc^4  =  \frac{1}{24} $ \\
&  $\vb^T \mA \vc^4 +  4 \vb^T \mAdot \vc^3 +  \bdot^T \vc^4  =  \frac{1}{30} $ \\
& $ \vb^T \left( \vc^2 \odot \mA^2 \vc \right) + 2 \bdot^T \left(\vc \odot \mA^2 \vc \right) +
 \vb^T \left(\vc^2 \odot \mA \vcdot \right) + \vb^T \left(\vc^2 \odot \mAdot \vc \right) + \bdot^T $ \\
  &	\; \; \; \; \; \; \;  $+ 2 \bdot^T \left(\vc \odot \mAdot \vc \right) 
  + \bdot^T \left( \vc^2 \odot \vcdot \right)  =  \frac{1}{36}$ \\	
& $ \vb^T \left(\vc \odot \mAdot^2 \vc^2 \right) + \bdot^T \mAdot^2 \vc^2 + \bdot^T \left( \vc \odot \mA \vc^2 \right) 
+ \vb^T \left( \vc \odot \mAdot \vc^2 \right)   $ \\
& \; \; \; \;   $  + 2 \vb^T \left( \vc \odot \mA \mAdot \vc \right) + \bdot^T \mAdot \vc^2 + 2 \bdot^T \mA \mAdot \vc 
+ 2 \bdot^T \left(\vc \odot \mAdot \vc \right) =\frac{1}{72} $\\
& $ \vb^T \mA^2 \vc^3 +  \bdot^T \mA \vc^3 +  \vb^T \mAdot \vc^3 +  3 \vb^T \mA \mAdot \vc^2 
+  3 \bdot^T \mAdot \vc^2  =  \frac{1}{120} $ \\
& $ \vb^T \left( \vc \odot \mA \vc \odot \mA \vc \right) 
+ \bdot^T \left(\mA \vc \odot \mA \vc \right)  
+ \vb^T \left(\vc \odot \mAdot \mA \vc \right)  $ \\	
& \; \; \; \; \; $ + \vb^T \left( \vc \odot \mA \left(\vc \odot \vcdot \right)\right)  
 + \bdot^T \left(\vc^2 \odot \mA \vc \right)  + \vb^T \left( \vc \odot \mAdot \vc^2\right) 	
+\bdot^T \mAdot \mA \vc   $ \\
& \; \; \; \; \; $ + \bdot^T \mA \left(\vc \odot \vcdot \right) 
+ \vb^T \left( \vc \odot  \mAdot  \vcdot \right) +  \bdot ^T \mAdot  c^2 
+  \bdot ^T \vc^2 \vcdot  +   \bdot ^T \mAdot  \vcdot   =  \frac{1}{48}   $ \\	
\end{tabular} 

\noindent \begin{tabular}{ll}
& $\vb^T \mA \left(\vc^2 \odot \mA \vc\right) + \vb^T \mA \left(\vc^2 \odot \vcdot \right)  
+ \bdot^T  \left(\vc^2 \odot \mA \vc \right) + 2 \vb^T \left( \mAdot \vc \odot \mA \vc\right) $ \\
&	\; \; \; \; \; $ + \vb^T \mAdot  \vc^3 + 2 \vb^T\left( \mAdot \vc \odot \vcdot \right) 
+ \bdot^T  \left( \vc^2 \odot \vcdot \right)  =  \frac{1}{60}$ \\
& $ \vb^T\left( \mA \vc \odot \mA \vc^2\right) +  \bdot^T \left( \vc \odot \mA \vc^2 \right) 
+ \vb^T \mAdot \mA \vc^2 + \vb^T \mAdot  \vc^3 + 2 \vb^T \left( \mA \vc \odot  \mAdot c\right) $ \\
& \; \; \; \; \; $ + 2 \vb^T \mAdot^2 \vc + 2 \bdot^T \left(\vc \odot  \mAdot \vc \right)  =  \frac{1}{90}   $   \\	          
& $ \vb^T \left(\vc \odot \mA^3 \vc \right) + \bdot^T \mA^3 \vc + \bdot^T \left( \vc \odot \mA^2 \vc \right)
+ \vb^T \left(  \vc \odot  \mAdot  \mA \vc\right)  $ \\
& \; \; \; \; \;  $ + \vb^T \left( \vc \odot \mA \mAdot \vc \right) 
 + \vb^T \left(\vc \odot \mA^2 \vcdot \right) + \bdot^T  \mAdot  \mA \vc  
+ \bdot^T A \mAdot \vc  + \bdot^T A^2\vcdot   $\\ 
&\; \; \; \; \;  $+ b^T\left(c \odot  \mAdot \vcdot \right) 
+ \bdot^T \mC \mA \vcdot  + \bdot^T \mC \mAdot \vc + \bdot^T  \mAdot \vcdot   =  \frac{1}{144} $\\	         	          	          	          
& $ \vb^T \left(  \mA \vc \odot \mA^2 \vc \right) + \vb^T \left(  \mA \vc \odot \mA \vcdot \right) 
+ \vb^T \left( \mA \vc \odot  \mAdot \vc \right) + \vb^T \left( \mAdot \vc \odot  \mA \vc\right) $ \\
& \; \; \; \; \;  $ + \vb^T \mAdot \mA^2 \vc + \bdot^T \left(\vc \odot \mA^2 \vc\right) 
 + \vb^T\left(  \mAdot \vc \odot \vcdot \right)  + \vb^T \mAdot \mA \vcdot   $ \\
& \; \; \; \; \;  $ + \bdot^T \left(\vc \odot A\vcdot \right) + \vb^T \mAdot ^2 \vc + \bdot^T \left(\vc \odot  \mAdot c\right)  =  \frac{1}{180} $\\
 & $ \vb^T( \mA^2 \left(\vc \odot  \mA \vc \right) + \vb^T \mA^2 \left(\vc \odot \vcdot \right) + \vb^T \mA \mAdot \vc^2 + \vb^T \mA \mAdot  \mA \vc $ \\
&   \; \; \; \; \;  $  + \vb^T \mAdot \left(\vc \odot  \mA \vc\right)  + \bdot^T \mA\left(\vc \odot  \mA \vc\right) + \vb^T \mA \mAdot \vcdot   + \vb^T \mAdot \left(\vc \odot \vcdot \right)  $ \\
&   \; \; \; \; \;  $  + \bdot^T \mA \left(\vc \odot \vcdot \right)  + \bdot^T  \mAdot \vc^2  + \bdot^T  \mAdot  \mA \vc + \bdot^T  \mAdot \vcdot   
=  \frac{1}{240} $\\
& $ \vb^T \mA^3 \vc^2 +  \bdot^T \mA^2 \vc^2 +  \vb^T \mAdot \mA \vc^2 +  \vb^T \mA \mAdot  \vc^2 
+  2 \vb^T \mA^2 \mAdot \vc +  \bdot^T  \mAdot  \vc^2 $        \\
& \; \; \; \; \;  $  +  2 \bdot^T \mA \mAdot \vc +  2 \vb^T \mAdot^2 \vc  =  \frac{1}{360}   $        \\
& $\vb^T \left(\vc \odot   \mA \vc \odot   \mA \vc \right) + \bdot^T \left( \mA \vc \odot   \mA \vc\right) 
+ 2 \vb^T \left(\vc \odot  \vcdot  \odot   \mA \vc\right) $ \\
& \; \; \; \; \;  $  + 2 \bdot^T \left(\vc^2 \odot   \mA \vc \right) + 2 \bdot^T \left(\vcdot  \odot   \mA \vc\right) + 2\bdot^T \left(c^2 \odot  \vcdot 	\right) 
 $ \\
 & \; \; \; \; \;  $ + \vb^T\left(\vc \odot  \vcdot ^2\right)   +  \bdot^T  \vcdot ^2  =  \frac{1}{24}     $ \\              
\end{tabular} 
}

\section{Order Conditions for IMEX two-derivative Runge--Kutta method} \label{sec:OC-IMEXMD}
The order conditions for a method \eqref{IMEX-RK}  are generally easier to formulate if the method is written in its Butcher form:
\begin{equation}\label{IMEX-RKmatrix1}
U = \ve u^n +   \Delta t \widehat{\mA} \F(U)+ \Delta t \mA \G(U) +  \Delta t^2  \dot{\mA} \Gdot(U). 
\end{equation}
The conversion between the two formulations (\ref{IMEX-RKmatrix})   and (\ref{IMEX-RKmatrix1}) is given by:
\begin{eqnarray}
\widehat{\mA}   =  \frac{1}{r} (I-\mP- \mW)^{-1} \mW, \;  \mA  =  \ (I-\mP- \mW)^{-1} \mD , \;  \dot{\mA}   =   (I-\mP- \mW)^{-1} \dot{\mD}.
\end{eqnarray}

The vectors $\widehat{\vb}$,  $\vb$, and $\dot{\vb}$ are given by the last row of $\widehat{\mA}$,  $\mA$, and $\dot{\mA}$, respectively.
The vectors $\vc=\mA \ve$, $\Cdot=\dot{\mA} \ve$, and $\Chat= \widehat{\mA} \ve$ define the time-levels at which the stages are happening;
these values are known as the abscissas.   The order conditions for methods of this form are:

\smallskip

{\renewcommand{\arraystretch}{1.1}
\begin{tabular}{llll} 
{For $p \geq 1$} & $\vb^t \ve=1$ &   \; \; \; & $\bhat^t \ve=1$ \\  
{For p $\geq$ 2}  & $\vb^t \vc + \bdot^t \ve=\frac{1}{2}$  &  &$\vb^t\Chat =\frac{1}{2} $ \\
			   & $\bhat^t \vc =\frac{1}{2} $   &. & $\bhat^t \Chat =\frac{1}{2}  $\\ 	   
{For $p \geq 3$}  & $ \vb^t\mA \vc + \bdot^t \vc + \vb^t \Cdot = \frac{1}{6}  $ & \;  & $ \vb^t\mA\Chat + \bdot^t\Chat = \frac{1}{6}  $ \\
 & $ \vb^t\widehat{\mA} \vc =\frac{1}{6} $  &  & $ \vb^t\widehat{\mA}\Chat = \frac{1}{6} $ \\ 
  & $ \bhat^t\mA \vc + \bhat^t \Cdot = \frac{1}{6} $  &  &  $ \bhat^t \mA \Chat  = \frac{1}{6} $ \\
 & $ \bhat^t\widehat{\mA} \vc =\frac{1}{6} $  &  &  $ \bhat^t\widehat{\mA}\Chat = \frac{1}{6} $ \\ 
 & $\vb^t(\vc \odot  \vc) + 2\bdot^t \vc = \frac{1}{3} $ &  &  $\vb^t( \vc \odot  \Chat) + \bdot^t \Chat = \frac{1}{3} $ \\
 & $\vb^t( \Chat  \odot  \Chat)  = \frac{1}{3} $ &  &  $\bhat^t(\vc \odot  \vc) = \frac{1}{3} $ \\
 & $\bhat^t( \vc \odot \Chat)  = \frac{1}{3} $ &  &  $\bhat^t( \Chat  \odot  \Chat)  = \frac{1}{3}$ \\ 
\end{tabular}

\smallskip

\section{Order Conditions for IMEX two-derivative GLM method} \label{sec:OC-IMEXGLM}

We wrote the  general linear method with $k$ steps and $s$ stages in a  matrix-vector notation \eqref{IMEX-GLMmatrix}
\begin{eqnarray*} 
 Y &= & \mR  U + \mP Y  +  \mW \left( Y + \frac{\Delta t}{r} \F(Y) \right)  +  \dt \mD \G(Y) +  \dt^2 \mDdot  \Gdot(Y), \\
 u^{n+1} &= &   \Gamma U + \mQ Y + \mV \left( Y + \frac{\Delta t}{r} \F(Y) \right). 
 \end{eqnarray*}
 It is more convenient to derive and present the order conditions in the form
\begin{eqnarray*}
Y & = &  \mT U^n + \Delta t  \mAh \F(Y) + \Delta t \mA \G(Y)  + \Delta t^2  \mAdot  \Gdot(Y) \\
u^{n+1} & = &  \theta U^n + \Delta t  \bhat \F(Y) + \Delta t  \vb \G(Y)  + \Delta t^2  \bdot \Gdot(Y) .
\end{eqnarray*}
where the conversion between the two sets of coefficients is given by
\[  \mT = (I - \mP - \mW)^{-1} \mR, \; \; \; \mAh = \frac{1}{r} (I - \mP - \mW)^{-1} \mW, \] 
\[ \mA =   (I - \mP - \mW)^{-1}  \mD, \; \; \; \mAdot =   (I - \mP - \mW)^{-1}  \mDdot, \]
\[ \theta =  \Gamma + (\mQ  + \mV) (I - \mP - \mW)^{-1} \mR, \; \; \; 
\bhat =   \frac{1}{r}  \left( (\mQ  + \mV) (I - \mP - \mW)^{-1} \mW + \mV \right) , \]
\[ \vb = (\mQ  + \mV) (I - \mP - \mW)^{-1} \mD, \; \; \;
\bdot = (\mQ  + \mV) (I - \mP - \mW)^{-1} \mDdot. \]

For a method to be order $P$ it must satisfy all the  order conditions  $p \leq P$. The 
conditions up to $P=3$  are given below, where we use the vector $\ell = \left[ 1-k, 2-k, . . ., 0\right]$.

{\renewcommand{\arraystretch}{1.15}
\begin{tabular}{llll} 
$p=1$ &  $ \theta \ell +  \bhat  e = 1 $ & \; \; \; &$  \theta \ell +  \vb  e = 1 $ \\
$p=2$ & $ \frac{1}{2}  \theta \ell^2 +  \bhat  (\mT \ell +  \mAh e)  = \frac{1}{2} $ &  &
$\frac{1}{2}  \theta \ell^2 +  \bhat  (\mT \ell +  \mA  e)  = \frac{1}{2} $  \\
& $ \frac{1}{2}  \theta \ell^2 + \vb  (\mT \ell + \mAh e)  = \frac{1}{2} $ &  & 
$\frac{1}{2}  \theta \ell^2   +  \vb  (\mT \ell +  \mA  e) +  {\dot\vb } e = \frac{1}{2}  $ \\
\end{tabular} \\
\begin{tabular}{ll} 
$p=3$ & $ \frac{1}{6}  \theta \ell^3 +
 \bhat  \left(  \frac{1}{2} \mT \ell^2 
+  \mAh (\mT \ell +  \mAh e) 
\right) = \frac{1}{6} $\\
& $\frac{1}{6}  \theta \ell^3 +
 \bhat  \left(  \frac{1}{2} \mT \ell^2 
+  \mAh (\mT \ell +  \mA  e) 
\right)
= \frac{1}{6} $ \\
& $ \frac{1}{6}  \theta \ell^3 +
 \bhat  \left(  \frac{1}{2} \mT \ell^2 
+  \mA  (\mT \ell +  \mAh e) \right)   = \frac{1}{6} $ \\
& $\frac{1}{6}  \theta \ell^3 +
 \bhat  \left(  \frac{1}{2} \mT \ell^2 
+  \mA  (\mT \ell +  \mA  e)
+  {\dot{\mA}} e \right)   = \frac{1}{6} $ \\
& $  \frac{1}{6}  \theta \ell^3 +
 \vb  \left(  \frac{1}{2} \mT \ell^2 
+  \mAh (\mT \ell +  \mAh e) 
\right) = \frac{1}{6} $ \\
& $  \frac{1}{6}  \theta \ell^3 +
 \vb  \left(  \frac{1}{2} \mT \ell^2 
+  \mAh (\mT \ell +  \mA  e) 
\right)
= \frac{1}{6} $ \\
& $  \frac{1}{6}  \theta \ell^3 +
 \vb  \left(  \frac{1}{2} \mT \ell^2 
+  \mA  (\mT \ell +  \mAh e) 
\right)  +  {\dot\vb } (\mT \ell +  \mAh e)
= \frac{1}{6} $ \\
&   $\frac{1}{6}  \theta \ell^3 +
 \vb  \left(  \frac{1}{2} \mT \ell^2 
+  \mA  (\mT \ell +  \mA  e) 
+  {\dot\mA } e \right) 
+  {\dot\vb } (\mT \ell +  \mA  e)
= \frac{1}{6} $ \\
& $  \frac{1}{3}  \theta \ell^3 +
 \bhat  \left( 
 (\mT \ell +  \mAh e) \odot
 (\mT \ell +  \mAh e) \right)
= \frac{1}{3} $ \\
& $  \frac{1}{3}  \theta \ell^3 +
 \bhat  \left( 
 (\mT \ell +  \mAh e) \odot
 (\mT \ell +  \mA  e) \right)
= \frac{1}{3} $ \\
& $ \frac{1}{3}  \theta \ell^3 +
 \bhat  \left( 
 (\mT \ell +  \mA  e) \odot
 (\mT \ell +  \mA  e) \right)
= \frac{1}{3} $ \\
& $   \frac{1}{3}  \theta \ell^3 +
 \vb  \left( 
 (\mT \ell +  \mAh e) \odot
 (\mT \ell +  \mAh e) \right)
= \frac{1}{3} $ \\
&  $\frac{1}{3}  \theta \ell^3 +
 \vb  \left( 
 (\mT \ell +  \mA  e) \odot
 (\mT \ell +  \mAh e) \right)
 +  {\dot{\vb}} 
 \left( \mT \ell +  \mAh e \right)
= \frac{1}{3} $ \\
& $ \frac{1}{3}  \theta \ell^3 +
 \vb  \left( 
 (\mT \ell +  \mA  e) \odot
 (\mT \ell +  \mA  e) \right)
 + 2  {\dot{\vb}} 
 \left( \mT \ell +  \mA  e \right)
= \frac{1}{3} $ \\
\end{tabular}

\bibliography{hu_bibtex,draft2}

\begin{thebibliography}{10}

\bibitem{tsai2014}
{\sc R.~P.~K.~C. A.~Y.~J.~Tsai and S.~Wang}, {\em Two-derivative
  {R}unge--{K}utta methods for {P}{D}{E}s using a novel discretization
  approach}, Numerical Algorithms, 65 (2014), pp.~687--703.

\bibitem{baiotti2005}
{\sc L.~Baiotti, I.~Hawke, P.~J. Montero, F.~Loffler, L.~Rezzolla,
  N.~Stergioulas, J.~A. Font, and E.~Seidel}, {\em Three-dimensional
  relativistic simulations of rotating neutron-star collapse to a {K}err black
  hole}, Physical Review D, 71 (2005).

\bibitem{balbas2005}
{\sc J.~Balb{\'a}s and E.~Tadmor}, {\em A central differencing simulation of
  the {O}rszag-{T}ang {V}ortex system}, IEEE Transactions on Plasma Science, 33
  (2005), pp.~470--471.

\bibitem{bassano2003}
{\sc E.~Bassano}, {\em Numerical simulation of thermo-solutal-capillary
  migration of a dissolving drop in a cavity}, IJNMF, 41 (2003), pp.~765--788.

\bibitem{BGK54}
{\sc P.~Bhatnagar, E.~Gross, and M.~Krook}, {\em A model for collision
  processes in gases. {I}. {S}mall amplitude processes in charged and neutral
  one-component systems}, Phys. Rev., 94 (1954), pp.~511--525.

\bibitem{BRAS2021113612}
{\sc M.~Bras, G.~Izzo, and Z.~Jackiewicz}, {\em A new class of strong stability
  preserving general linear methods}, Journal of Computational and Applied
  Mathematics, 396 (2021), p.~113612.

\bibitem{msrk}
{\sc Z.~G.~D.~H. D.~I.~K. C.~Bresten, S.~Gottlieb and A.~N{\'e}meth}, {\em
  Strong stability preserving multistep {R}unge-{K}utta methods}, Mathematics
  of Computation, 86 (2017), pp.~747--769.

\bibitem{caiden2001}
{\sc R.~Caiden, R.~P. Fedkiw, and C.~Anderson}, {\em A numerical method for
  two-phase flow consisting of separate compressible and incompressible
  regions}, Journal of Computational Physics, 166 (2001), pp.~1--27.

\bibitem{carrillo2003}
{\sc J.~Carrillo, I.~M. Gamba, A.~Majorana, and C.-W. Shu}, {\em A weno-solver
  for the transients of boltzmann-poisson system for semiconductor devices:
  performance and comparisons with monte carlo methods}, Journal of
  Computational Physics, 184 (2003), pp.~498--525.

\bibitem{tsai2010}
{\sc R.~P.~K. Chan and A.~Y.~J. Tsai}, {\em On explicit two-derivative
  {R}unge-{K}utta methods}, Numerical Algorithms, 53 (2010), pp.~171--194.

\bibitem{cheng2003}
{\sc L.-T. Cheng, H.~Liu, and S.~Osher}, {\em Computational high-frequency wave
  propagation using the level set method, with applications to the
  semi-classical limit of {S}chrodinger equations}, Comm. Math. Sci., 1 (2003),
  pp.~593--621.

\bibitem{cheruvu2007}
{\sc V.~Cheruvu, R.~D. Nair, and H.~M. Turfo}, {\em A spectral finite volume
  transport scheme on the cubed-sphere}, Applied Numerical Mathematics, 57
  (2007), pp.~1021--1032.

\bibitem{SDpaper}
{\sc A.~Christlieb, S.~Gottlieb, Z.~Grant, and D.~C. Seal}, {\em Explicit
  strong stability preserving multistage two-derivative time-stepping schemes},
  Journal of Scientific Computing, 68(3) (2016), pp.~914--942.

\bibitem{cockburn2004}
{\sc B.~Cockburn, F.~Li, and C.-W. Shu}, {\em Locally divergence-free
  discontinuous {G}alerkin methods for the {M}axwell equations}, Journal of
  Computational Physics, 194 (2004), pp.~588--610.

\bibitem{cockburn2005}
{\sc B.~Cockburn, J.~Qian, F.~Reitich, and J.~Wang}, {\em An accurate
  spectral/discontinuous finite-element formulation of a phase-space-based
  level set approach to geometrical optics}, Journal of Computational Physics,
  208 (2005), pp.~175--195.

\bibitem{sealMSMD2014}
{\sc Y.~G. D.~C.~Seal and A.~J. Christlieb}, {\em High-order multiderivative
  time integrators for hyperbolic conservation laws}, Journal of Scientific
  Computing, 60 (2014), pp.~101--140.

\bibitem{tsrk}
{\sc S.~G. D.~I.~Ketcheson and C.~B. Macdonald}, {\em Strong stability
  preserving two-step {R}unge-{K}utta methods}, SIAM Journal on Numerical
  Analysis, 49 (2012), pp.~2618--2639.

\bibitem{DaruTenaud}
{\sc V.~Daru and C.~Tenaud}, {\em High order one-step monotonicity-preserving
  schemes for unsteady compressible flow calculations}, Journal of
  Computational Physics, 193 (2004), pp.~563--594.

\bibitem{delzanna2002}
{\sc L.~{Del Zanna} and N.~Bucciantini}, {\em An efficient shock-capturing
  central-type scheme for multidimensional relativistic flows: I.
  hydrodynamics}, Astronomy and Astrophysics, 390 (2002), pp.~1177--1186.

\bibitem{EIS2dSSP2020}
{\sc A.~Ditkowski, S.~Gottlieb, and Z.~J. Grant}, {\em Two-derivative error
  inhibiting schemes and enhanced error inhibiting schemes}, SIAM Journal on
  Numerical Analysis, 58 (2020), pp.~3197--3225.

\bibitem{LiDu2018}
{\sc Z.~Du and J.~Li}, {\em A hermite weno reconstruction for fourth order
  temporal accurate schemes based on the grp solver for hyperbolic conservation
  laws}, Journal of Computational Physics, 355 (2018), pp.~385--396.

\bibitem{enright2002}
{\sc D.~Enright, R.~Fedkiw, J.~Ferziger, and I.~Mitchell}, {\em A hybrid
  particle level set method for improved interface capturing}, Journal of
  Computational Physics, 183 (2002), pp.~83--116.

\bibitem{feng2004}
{\sc L.-L. Feng, C.-W. Shu, and M.~Zhang}, {\em A hybrid cosmological
  hydrodynamic/n-body code based on a weighted essentially nonoscillatory
  scheme}, The Astrophysical Journal, 612 (2004), pp.~1--13.

\bibitem{ferracina2005a}
{\sc L.~Ferracina and M.~Spijker}, {\em Stepsize restrictions for
  total-variation-boundedness in general runge-kutta procedures}, Applied
  Numerical Mathematics, 53 (2005), pp.~265--279.

\bibitem{ferracina2004}
{\sc L.~Ferracina and M.~N. Spijker}, {\em Stepsize restrictions for the
  total-variation-diminishing property in general {R}unge-{K}utta methods},
  SIAM Journal of Numerical Analysis, 42 (2004), pp.~1073--1093.

\bibitem{ferracina2005}
\leavevmode\vrule height 2pt depth -1.6pt width 23pt, {\em An extension and
  analysis of the {S}hu-{O}sher representation of {R}unge-{K}utta methods},
  Mathematics of Computation, 249 (2005), pp.~201--219.

\bibitem{ferracina2008}
{\sc L.~Ferracina and M.~Spjker}, {\em Strong stability of
  singly-diagonally-implicit {R}unge-{K}utta methods}, Applied Numerical
  Mathematics,  (2008).
\newblock to appear.

\bibitem{gottlieb2003}
{\sc S.~Gottlieb and L.-A.~J. Gottlieb}, {\em Strong stability preserving
  properties of {R}unge-{K}utta time discretization methods for linear constant
  coefficient operators}, Journal of Scientific Computing, 18 (2003),
  pp.~83--109.

\bibitem{Jingwei2022}
{\sc S.~Gottlieb, Z.~J. Grant, J.~Hu, and R.~Shu}, {\em High order strong
  stability preserving multiderivative implicit and imex runge--kutta methods
  with asymptotic preserving properties}, SIAM Journal on Numerical Analysis,
  60 (2022), pp.~423--449.

\bibitem{SSPbook2011}
{\sc S.~Gottlieb, D.~Ketcheson, and C.-W. Shu}, {\em Strong Stability
  Preserving Runge-Kutta and Multistep Time Discretizations}, World Scientific,
  2011.

\bibitem{gottlieb1998}
{\sc S.~Gottlieb and C.-W. Shu}, {\em Total variation diminishing
  {R}unge-{K}utta schemes}, Mathematics of Computation, 67 (1998), pp.~73--85.

\bibitem{gottlieb2001}
{\sc S.~Gottlieb, C.-W. Shu, and E.~Tadmor}, {\em Strong stability preserving
  high-order time discretization methods}, SIAM Review, 43 (2001), pp.~89--112.

\bibitem{TSpaper}
{\sc Z.~Grant, S.~Gottlieb, and D.~Seal}, {\em A strong stability preserving
  analysis for explicit multistage two-derivative time-stepping schemes based
  on taylor series conditions}, Communications on Applied Mathematics and
  Computation, 1 (2019), pp.~21--59.

\bibitem{YiannisDW}
{\sc Y.~Hadjimichael}, 2017.

\bibitem{higueras2004a}
{\sc I.~Higueras}, {\em On strong stability preserving time discretization
  methods}, Journal of Scientific Computing, 21 (2004), pp.~193--223.

\bibitem{higueras2005a}
\leavevmode\vrule height 2pt depth -1.6pt width 23pt, {\em Representations of
  {R}unge-{K}utta methods and strong stability preserving methods}, Siam
  Journal On Numerical Analysis, 43 (2005), pp.~924--948.

\bibitem{hundsdorfer2005}
{\sc W.~Hundsdorfer and S.~J. Ruuth}, {\em On monotonicity and boundedness
  properties of linear multistep methods}, Mathematics of Computation, 75
  (2005), pp.~655--672.

\bibitem{hundsdorfer2003}
{\sc W.~Hundsdorfer, S.~J. Ruuth, and R.~J. Spiteri}, {\em
  Monotonicity-preserving linear multistep methods}, SIAM Journal of Numerical
  Analysis, 41 (2003), pp.~605--623.

\bibitem{Izzo2015}
{\sc G.~Izzo and Z.~Jackiewicz},  (2015).

\bibitem{IZZO2020206}
{\sc G.~Izzo and Z.~Jackiewicz}, {\em Strong stability preserving
  implicit–explicit transformed general linear methods}, Mathematics and
  Computers in Simulation, 176 (2020), pp.~206--225.
\newblock Applied Scientific Computing XV: Innovative Modelling and Simulation
  in Sciences.

\bibitem{QiuDumbserShu}
{\sc M.~D. J.~Qiu and C.-W. Shu}, {\em The discontinuous galerkin method with
  lax--wendroff type time discretizations}, Computer Methods in Applied
  Mechanics and Engineering, 194 (2005), pp.~4528--4543.

\bibitem{jin2005}
{\sc S.~Jin, H.~Liu, S.~Osher, and Y.-H.~R. Tsai}, {\em Computing multivalued
  physical observables for the semiclassical limit of the {S}chrodinger
  equation}, Journal of Computational Physics, 205 (2005), pp.~222--241.

\bibitem{KaWa72}
{\sc K.~Kastlunger and G.~Wanner}, {\em On {T}uran type implicit
  {R}unge-{K}utta methods}, Computing (Arch. Elektron. Rechnen), 9 (1972),
  pp.~317--325.

\bibitem{KaWa72-RK}
{\sc K.~H. Kastlunger and G.~Wanner}, {\em Runge {K}utta processes with
  multiple nodes}, Computing (Arch. Elektron. Rechnen), 9 (1972), pp.~9--24.

\bibitem{ketcheson2008}
{\sc D.~I. Ketcheson}, {\em Highly efficient strong stability preserving
  {R}unge-{K}utta methods with low-storage implementations}, SIAM Journal on
  Scientific Computing,  (2008).
\newblock to appear.

\bibitem{KetchDW}
\leavevmode\vrule height 2pt depth -1.6pt width 23pt, {\em Step sizes for
  strong stability preserving with downwind-biased operators}, SIAM Journal on
  Numerical Analysis, 49 (2011), pp.~1649--1660.

\bibitem{ketcheson2007}
{\sc D.~I. Ketcheson, C.~B. Macdonald, and S.~Gottlieb}, {\em Optimal implicit
  strong stability preserving {R}unge-{K}utta methods}, Applied Numerical
  Mathematics.
\newblock submitted.

\bibitem{kraaijevanger1991}
{\sc J.~F. B.~M. Kraaijevanger}, {\em Contractivity of {R}unge-{K}utta
  methods}, BIT, 31 (1991), pp.~482--528.

\bibitem{labrunie2004}
{\sc S.~Labrunie, J.~Carrillo, and P.~Bertrand}, {\em Numerical study on
  hydrodynamic and quasi-neutral approximations for collisionless two-species
  plasmas}, Journal of Computational Physics, 200 (2004), pp.~267--298.

\bibitem{LiDu2016a}
{\sc J.~Li and Z.~Du}, {\em A two-stage fourth order time-accurate
  discretization for lax-wendroff type flow solvers i. hyperbolic conservation
  laws}, SIAM J. Sci. Computing, 38 (2016), pp.~3046--3069.

\bibitem{liu2008}
{\sc Y.~Liu, C.-W. Shu, and M.~Zhang}, {\em Strong stability preserving
  property of the deferred correction time discretization}, Journal of
  Computational Mathematics.
\newblock to appear.

\bibitem{DumbserBalsara1}
{\sc A.~H. M.~Dumbser, O.~Zanotti and D.~S. Balsara}, {\em Ader-weno finite
  volume schemes with space-time adaptive mesh refinement}, Journal of
  Computational Physics, 248 (2013), pp.~257 -- 286.

\bibitem{mignone2005}
{\sc A.~Mignone}, {\em The dynamics of radiative shock waves: linear and
  nonlinear evolution}, The Astrophysical Journal, 626 (2005), pp.~373--388.

\bibitem{mitsui1982}
{\sc T.~Mitsui}, {\em {R}unge-{K}utta type integration formulas including the
  evaluation of the second derivative. i.}, Publ. Res. Inst. Math. Sci.

\bibitem{MORADI2019}
{\sc A.~Moradi, J.~Farzi, and A.~Abdi}, {\em Strong stability preserving second
  derivative general linear methods}, Journal of Scientific Computing, 81
  (2019), pp.~392 -- 435.

\bibitem{Nguyen-Ba2010}
{\sc T.~Nguyen-Ba, H.~Nguyen-Thu, T.~Giordano, and R.~Vaillancourt}, {\em
  One-step strong-stability-preserving hermite-birkhoff-taylor methods},
  Scientific Journal of Riga Technical University, 45 (2010), pp.~95--104.

\bibitem{nguyen2014strong}
{\sc H.~Nguyen-Thu, T.~Nguyen-Ba, and R.~Vaillancourt}, {\em
  Strong-stability-preserving, hermite--birkhoff time-discretization based on k
  step methods and 8-stage explicit runge--kutta methods of order 5 and 4},
  Journal of computational and applied mathematics, 263 (2014), pp.~45--58.

\bibitem{ono2004}
{\sc H.~Ono and T.~Yoshida}, {\em Two-stage explicit {R}unge-{K}utta type
  methods using derivatives.}, Japan J. Indust. Appl. Math., 21 (2004),
  pp.~61--374.

\bibitem{LiDu2016b}
{\sc L.~Pan, K.~Xu, Q.~Li, and J.~Li}, {\em An efficient and accurate two-stage
  fourth-order gas-kinetic scheme for the euler and navier--stokes equations},
  Journal of Computational Physics.

\bibitem{pantano2007}
{\sc C.~Pantano, R.~Deiterding, D.~Hill, and D.~Pullin}, {\em A low numerical
  dissipation patch-based adaptive mesh refinement method for large-eddy
  simulation of compressible flows}, Journal of Computational Physics, 221
  (2007), pp.~63--87.

\bibitem{patel2005}
{\sc S.~Patel and D.~Drikakis}, {\em Effects of preconditioning on the accuracy
  and efficiency of incompressible flows}, IJNMF, 47 (2005), pp.~963--970.

\bibitem{peng1999}
{\sc D.~Peng, B.~Merriman, S.~Osher, H.~Zhao, and M.~Kang}, {\em A pde-based
  fast local level set method}, Journal of Computational Physics, 155 (1999),
  pp.~410--438.

\bibitem{QIN2023}
{\sc X.~Qin, Z.~Jiang, and J.~Yu}, {\em Strong stability-preserving
  three-derivative runge–kutta methods}, Computational and Applied
  Mathematics, 42 (2023), p.~171.

\bibitem{QIN2024106089}
{\sc X.~Qin, J.~Yu, Z.~Jiang, L.~Huang, and C.~Yan}, {\em Explicit strong
  stability preserving second derivative multistep methods for the euler and
  navier–stokes equations}, Computers \& Fluids, 268 (2024), p.~106089.

\bibitem{ruuth2006}
{\sc S.~Ruuth}, {\em Global optimization of explicit
  strong-stability-preserving {R}unge-{K}utta methods}, Math. Comp., 75 (2006),
  pp.~183--207.

\bibitem{ruuth2001}
{\sc S.~J. Ruuth and R.~J. Spiteri}, {\em Two barriers on
  strong-stability-preserving time discretization methods}, Journal of
  Scientific Computation, 17 (2002), pp.~211--220.

\bibitem{ruuth2004}
{\sc S.~J. Ruuth and R.~J. Spiteri}, {\em High-order
  strong-stability-preserving {R}unge-{K}utta methods with downwind-biased
  spatial discretizations}, SIAM Journal of Numerical Analysis, 42 (2004),
  pp.~974--996.

\bibitem{shintani1971}
{\sc H.~Shintani}, {\em On one-step methods utilizing the second derivative},
  Hiroshima Mathematical Journal,.

\bibitem{shintani1972}
\leavevmode\vrule height 2pt depth -1.6pt width 23pt, {\em On explicit one-step
  methods utilizing the second derivative}, Hiroshima Mathematical Journali, 2
  (1972), pp.~353--368.

\bibitem{shu2002}
{\sc C.-W. Shu}, {\em A survey of strong stability-preserving high-order time
  discretization methods}, in Collected Lectures on the Preservation of
  Stability under discretization, SIAM: Philadelphia, PA, 2002.

\bibitem{shu1988}
{\sc C.-W. Shu and S.~Osher}, {\em Efficient implementation of essentially
  non-oscillatory shock-capturing schemes}, Journal of Computational Physics,
  77 (1988), pp.~439--471.

\bibitem{shu1988b}
\leavevmode\vrule height 2pt depth -1.6pt width 23pt, {\em Efficient
  implementation of essentially non-oscillatory shock-capturing schemes},
  Journal of Computational Physics, 77 (1988), p.~439–471.

\bibitem{spijker2007}
{\sc M.~N. Spijker}, {\em Stepsize conditions for general monotonicity in
  numerical initial value problems}, Siam Journal On Numerical Analysis, 45
  (2007), pp.~1226--1245.

\bibitem{spiteri2002}
{\sc R.~J. Spiteri and S.~J. Ruuth}, {\em A new class of optimal high-order
  strong-stability-preserving time discretization methods}, SIAM Journal of
  Numerical Analysis, 40 (2002), pp.~469--491.

\bibitem{spiteri2003}
\leavevmode\vrule height 2pt depth -1.6pt width 23pt, {\em Nonlinear evolution
  using optimal fourth-order strong-stability-preserving {R}unge-{K}utta
  methods}, Mathematics and Computers in Simulation, 62 (2003), pp.~125--135.

\bibitem{StSt63}
{\sc D.~D. Stancu and A.~H. Stroud}, {\em Quadrature formulas with simple
  {G}aussian nodes and multiple fixed nodes}, Math. Comp., 17 (1963),
  pp.~384--394.

\bibitem{sun2006}
{\sc Y.~Sun, Z.~Wang, and Y.~Liu}, {\em Spectral (finite) volume method for
  conservation laws on unstructured grids {VI}: Extension to viscous flow},
  Journal of Computational Physics, 215 (2006), pp.~41--58.

\bibitem{tanguay2003}
{\sc M.~Tanguay and T.~Colonius}, {\em Progress in modeling and simulation of
  shock wave lithotripsy (swl)}, in Fifth International Symposium on cavitation
  (CAV2003), no.~OS-2-1-010, 2003.

\bibitem{Tu50}
{\sc P.~Tur{\'a}n}, {\em On the theory of the mechanical quadrature}, Acta
  Sci.Math. Szeged, 12 (1950), pp.~30--37.

\bibitem{wang2005}
{\sc Z.~Wang and Y.~Liu}, {\em The spectral difference method for the 2{D}
  {E}uler equations on unstructured grids}, in 17th AIAA Computational Fluid
  Dynamics Conference, AIAA, 2005.

\bibitem{wang2007a}
{\sc Z.~Wang, Y.~Liu, G.~May, and A.~Jameson}, {\em Spectral difference method
  for unstructured grids {II}: Extension to the {E}uler equations}, Journal of
  Scientific Computing, 32 (2007), pp.~45--71.

\bibitem{zhang2006}
{\sc W.~Zhang and A.~I. MacFayden}, {\em {RAM}: A relativistic adaptive mesh
  refinement hydrodynamics code}, The Astrophysical Journal Supplement Series,
  164 (2006), pp.~255--279.

\end{thebibliography}

\end{document}